\documentclass[12pt]{amsart}
\usepackage{svn}
\SVN $Revision: 164 $

\usepackage{amssymb,amsthm,amsmath,amsfonts}
\usepackage{graphicx}
\usepackage{subcaption}
\usepackage{epstopdf}
\usepackage{mathtools}
\usepackage{mathabx}
\usepackage{blkarray}
\usepackage{comment}

\usepackage{textcomp}

\usepackage{enumerate}
\usepackage{enumitem}
\usepackage{commath}
\usepackage{color}

\usepackage[colorlinks=true,urlcolor=black,anchorcolor=black,citecolor=black,linkcolor=black]{hyperref}
\usepackage[hypcap]{caption}
\usepackage[a4paper,margin=1in]{geometry}
\usepackage{tikz,pgfplots,pgfplotstable}
\pgfplotsset{compat=1.16}

\usetikzlibrary{arrows.meta}
\usetikzlibrary{shapes,patterns}
\definecolor{histcolor}{rgb}{0.97858, 0.678934, 0.157834}

\newcommand{\paczek}[3][white]{  \node[draw,ellipse,inner sep=1pt, fill=#1] at (#2){\phantom{x}$#3$\phantom{x}};}

\allowdisplaybreaks{}

\DeclareMathOperator{\id}{\operatorname{id}}

\newcommand{\IE}{E}

\DeclareMathOperator{\IntPart}{\mathit{Int}}

\DeclareFontFamily{U} {MnSymbolC}{}
\DeclareFontShape{U}{MnSymbolC}{m}{n}{
  <-6> MnSymbolC5
  <6-7> MnSymbolC6
  <7-8> MnSymbolC7
  <8-9> MnSymbolC8
  <9-10> MnSymbolC9
  <10-12> MnSymbolC10
  <12-> MnSymbolC12}{}
\DeclareFontShape{U}{MnSymbolC}{b}{n}{
  <-6> MnSymbolC-Bold5
  <6-7> MnSymbolC-Bold6
  <7-8> MnSymbolC-Bold7
  <8-9> MnSymbolC-Bold8
  <9-10> MnSymbolC-Bold9
  <10-12> MnSymbolC-Bold10
  <12-> MnSymbolC-Bold12}{}

\DeclareSymbolFont{MnSyC} {U} {MnSymbolC}{m}{n}

\DeclareMathSymbol{\oocirc}{\mathbin}{MnSyC}{99}

\newcommand{\bPsi}{\boldsymbol{\Psi{}}}

\def\C{{\mathbb C}}
\newcommand\IC{{\mathbb C}}

\newcommand\IZ{{\mathbb Z}}
\newcommand\IN{{\mathbb N}}

\newcommand\bub[1]{\mathring{#1}}

\def\<{\langle}
\def\>{\rangle}

\def\E{E}

\DeclareMathOperator{\NC}{\mathit{NC}} \DeclareMathOperator{\NCirr}{{\mathit{NC}^\textit{irr}}}

\DeclareMathOperator{\IB}{\mathrm{inn}} \DeclareMathOperator{\OB}{\mathrm{out}}

 \newcommand{\alg}[1]{\mathcal{#1}}

\renewcommand{\abs}[1]{\lvert#1\rvert}

\newcommand{\harpleftsign}{\scriptstyle\leftharpoonup}
\newcommand{\harpleft}[2]{\ifx\displaystyle#1\doalign{$\harpleftsign$}{#1#2}\fi
  \ifx\textstyle#1\doalign{$\harpleftsign$}{#1#2}\fi
  \ifx\scriptstyle#1\doalign{\scalebox{.6}[.9]{$\harpleftsign$}}{#1#2}\fi
  \ifx\scriptscriptstyle#1\doalign{\scalebox{.5}[.8]{$\harpleftsign$}}{#1#2}\fi
}

\newcommand{\harprightsign}{\scriptstyle\rightharpoonup}
\newcommand{\harpright}[2]{\ifx\displaystyle#1\doalign{$\harprightsign$}{#1#2}\fi
  \ifx\textstyle#1\doalign{$\harprightsign$}{#1#2}\fi
  \ifx\scriptstyle#1\doalign{\scalebox{.6}[.9]{$\harprightsign$}}{#1#2}\fi
  \ifx\scriptscriptstyle#1\doalign{\scalebox{.5}[.8]{$\harprightsign$}}{#1#2}\fi
}
\newcommand{\doalign}[2]{{\vbox{\offinterlineskip\ialign{\hfil##\hfil\cr#1\cr$#2$\cr}}}}

\usepackage{yhmath}

\newcommand{\hl}[1]{\mathpalette\harpleft{#1}}
\newcommand{\hr}[1]{\mathpalette\harpright{#1}}

\DeclareMathOperator{\rE}{\mathpalette\harpright{E}}
\DeclareMathOperator{\lE}{\mathpalette\harpleft{E}}

\makeatletter
 \newcommand*{\closeindex}[1]{_{\mkern-4.5mu#1}}

 \DeclareRobustCommand{\rdelta}{\mathpalette\harpright{\delta}\@ifnextchar_{\expandafter\closeindex\@gobble}{}}
 \DeclareRobustCommand{\ldelta}{\mathpalette\harpleft{\delta}\@ifnextchar_{\expandafter\closeindex\@gobble}{}}

 \DeclareRobustCommand{\rnabla}{\hr{\nabla}\@ifnextchar_{\expandafter\closeindex\@gobble}{}}
 \DeclareRobustCommand{\lnabla}{\hl{\nabla}\@ifnextchar_{\expandafter\closeindex\@gobble}{}}
 \makeatother

\makeatletter
\renewcommand{\tocsection}[3]{\indentlabel{\@ifnotempty{#2}{\bfseries\ignorespaces#1 #2\quad}}\bfseries#3}
\renewcommand{\tocsubsection}[3]{\indentlabel{\@ifnotempty{#2}{\ignorespaces#1 #2\quad}}#3}

\newcommand\@dotsep{4.5}
\def\@tocline#1#2#3#4#5#6#7{\relax
	\ifnum #1>\c@tocdepth \else
	\par \addpenalty\@secpenalty\addvspace{#2}\begingroup \hyphenpenalty\@M
	\@ifempty{#4}{\@tempdima\csname r@tocindent\number#1\endcsname\relax
	}{\@tempdima#4\relax
	}\parindent\z@ \leftskip#3\relax \advance\leftskip\@tempdima\relax
	\rightskip\@pnumwidth plus1em \parfillskip-\@pnumwidth
	#5\leavevmode\hskip-\@tempdima{#6}\nobreak
	\leaders\hbox{$\m@th\mkern \@dotsep mu\hbox{.}\mkern \@dotsep mu$}\hfill
	\nobreak
	\hbox to\@pnumwidth{\@tocpagenum{\ifnum#1=1\bfseries\fi#7}}\par \nobreak
	\endgroup
	\fi}
\AtBeginDocument{\expandafter\renewcommand\csname r@tocindent0\endcsname{0pt}
}
\def\l@subsection{\@tocline{2}{0pt}{2.5pc}{5pc}{}}
\makeatother

\newcommand{\bbeta}[1]{\beta^b_{#1}}

\newcommand{\bbetas}[2]{\beta^{b,#1}_{#2}}
\newcommand{\bubbetas}[2]{\bub{\beta}^{b,#1}_{#2}}

\newcommand{\fbeta}[1]{\beta^\delta_{#1}}

\newcommand{\fbetas}[2]{\beta^{\delta,#1}_{#2}}

\usepackage{xargs}                      \usepackage{xcolor}  \usepackage[colorinlistoftodos,textsize=tiny]{todonotes}
\usepackage{todonotes}
\newcommandx{\question}[2][1=]{\todo[linecolor=blue,backgroundcolor=blue!25,bordercolor=blue,#1]{#2}}
\newcommandx{\adrian}[2][1=]{\todo[linecolor=teal,backgroundcolor=teal!25,bordercolor=teal,#1]{#2}}
\newcommandx{\franz}[2][1=]{\todo[linecolor=teal,backgroundcolor=green!25,bordercolor=teal,#1]{#2}}
\newcommandx{\kamil}[2][1=]{\todo[linecolor=teal,backgroundcolor=purple!25,bordercolor=teal,#1]{#2}}
\newcommandx{\info}[2][1=]{\todo[caption={Short note},linecolor=violet,backgroundcolor=violet!25,bordercolor=violet,size=\tiny,#1]{#2}}
\newcommandx{\thiswillnotshow}[2][1=]{\todo[disable,#1]{#2}}

\newcommand{\be}{\begin{equation}}
\newcommand{\ee}{\end{equation}}
\newtheorem{theorem}{Theorem}[section]
\newtheorem{proposition}[theorem]{Proposition}
\newtheorem{corollary}[theorem]{Corollary}
\newtheorem{lemma}[theorem]{Lemma}
       
\theoremstyle{definition}
\newtheorem{notation}[theorem]{Notation}
\newtheorem{definition}[theorem]{Definition}

\newtheorem{remark}[theorem]{Remark}

\newtheorem{example}[theorem]{Example}
\newtheorem{remdef}[theorem]{Remark and Definition}
\newtheorem{remnot}[theorem]{Remark and Notation}

\title[Polynomials in $c$-free random variables]{Polynomials in $c$-free random variables with applications to free denoising}

\author[A. Celestino]{Adrián Celestino}
\address[A. Celestino]{Mathematisches Institut\\
	Universit\"at M\"unster\\
	Einsteinstr. 62\\
	48149 M\"unster, Germany\\
	ORCID: 0000-0003-3444-5000
}
\email{adrian.celestino@uni-muenster.de}

\author[F. Lehner]{Franz Lehner}
\address[F. Lehner]{Institut f\"ur Diskrete Mathematik\\
Technische Universit\"at Graz\\
Steyrergasse 30\\
A-8010 Graz, Austria\\
ORCID: 0000-0002-6902-5148
}
\email{lehner@math.tu-graz.ac.at}

\author[K. Szpojankowski]{Kamil Szpojankowski}
\address[K. Szpojankowski]{Institute of Mathematics of the Polish Academy of Sciences\\ ul. \'Sniadeckich 8\\ 00-656 Warszawa\\ Poland\\ and 
	Faculty of Mathematics and Information Science\\
	Warsaw University of Technology\\ Koszykowa 75\\ 00-662 Warszawa\\ Poland.\\
ORCID: 0000-0003-0926-1285
}
\email{kamil.szpojankowski@pw.edu.pl}

\thanks{Supported by the Austrian Federal Ministry of Education, Science and
  Research and the Polish Ministry of Science and Higher Education, grants
  N$^{\textrm{os}}$ PL 08/2016 and PL 06/2018.\\
  AC: Funded by the Deutsche Forschungsgemeinschaft (DFG, German Research Foundation) under Germany's Excellence Strategy EXC 2044/2 --390685587, Mathematics Münster: Dynamics--Geometry--Structure.\\
  FL: This research was funded in part by the Austrian Science Fund (FWF) grant
  I 6232-N BOOMER (WEAVE).\\
KSz: This research was funded in part by National Science Centre, Poland WEAVE-UNISONO grant BOOMER 2022/04/Y/ST1/00008.
}

\numberwithin{equation}{section}

\begin{document}
\date{Rev. \SVNRevision, \today}

\begin{abstract}
  We study distributions of polynomials in conditionally free ($c$-free) random
  variables, a notion of independence
  for
two-state noncommutative probability spaces $(\mathcal{A},\varphi,\psi)$ introduced by
  Bo\.zejko, Leinert and Speicher.
  To this end we establish recursive relations between the
  joint Boolean cumulants of $c$-free random variables,
  analogous to previously found recursions for Boolean cumulants of free random variables.
  The algebraic reformulation of these recursions on the free associative
  algebra provides an effective formal machinery for the computation
  of the moment  generating functions and thus  the   distributions
  of arbitrary self-adjoint polynomials in $c$-free random variables.  
   
  As an application of a recent observation, our approach can be used to determine conditional expectations
  of the form $\E\left[a|P(a,b)\right]$, where $P$ is a self-adjoint polynomial
  in free (in the sense of Voiculescu) random variables $a,b$.
  We illustrate this with an example where $P(a,b)=i[a,b]$.  
  
  Finally we define orthogonal projections
  that formally play the role of conditional expectations in the framework of
  $c$-freeness and share some properties with the conditional expectations of
  free variables. In particular they can be used to re-derive by purely
  algebraic methods the formula of Popa and Wang for the $\Sigma$-transform for
  the $c$-free multiplicative convolution.
\end{abstract}

\maketitle

\begin{flushright}
	\itshape
	Dedicated to professor Marek Bo\.zejko\\ on the occasion of his 80th birthday.
\end{flushright}
\tableofcontents{}

\section{Introduction}

The notion of \(c\)-freeness (or conditional freeness) was introduced by
Bo\.{z}ejko, Leinert, and Speicher in \cite
{BozejkoSpeicher:psi:1991,BozejkoLeinertSpeicher:1996:convolution}. It is a
noncommutative independence in the framework of a two-state noncommutative
probability space \((\alg{M}, \varphi, \psi)\). This concept generalizes
freeness in the following sense: random variables \(a,b \in \alg{M}\) are
\(c\)-free with respect to the pair \((\varphi,\psi)\), if they are free in the sense of
Voiculescu with respect to the second state \(\psi\) and satisfy
a factorization condition with respect to the first state $\varphi$.
For precise definitions of both types of independence see Section~\ref{sec:2} below.

The distribution of a self-adjoint element \(a\) in a two-state probability
space is described by a pair of measures \(\mu_a^{\varphi}\) and
\(\mu_a^{\psi}\), corresponding to the distributions with respect to
\(\varphi\) and \(\psi\), respectively. In this paper, we present a method for
determining the distribution of polynomials in \(c\)-free variables. More
precisely, for \(c\)-free variables \(a,b\) and a noncommutative polynomial \(P
\in \C\langle X,Y\rangle\), we provide a method for determining the (ordinary) moment
generating function of \(P(a,b)\) with respect to \(\varphi\) (see Theorem
\ref{thm:cfreeConvolution}). For simplicity, our presentation focuses on the
case of two variables, but this generalizes easily to the case of $n$
variables. Our main tools are the same as in
\cite{LehnerSzpojankowski:2023:freeint1}
where we studied polynomials in free variables:
\begin{itemize}
 \item
  derivations on the free algebra which translate combinatorial
  recursions for Boolean cumulants into effective algebraic identities
  between their generating functions;
 \item
  linearization, which allows us to express the resolvent
  $(1-zP(X,Y))^{-1}$ in the form $u^t (I-zAX-zBY)^{-1}v$ for vectors
  $u,v\in \IC^n$ and $n\times n$ matrices $A$ and $B$.
\end{itemize}
On the combinatorial side of this development we prove a new characterization
of $c$-freeness in terms of Boolean cumulants (see Theorem
\ref{thm:VNTPCfree}). It turns out to be very similar to the characterization
of freeness from \cite{FMNS2}, and involves the so-called Vertical-No-Repeat
property which identifies maximal elements among coloured noncrossing
partitions with respect to a partial order denoted by $\ll$, which arises
naturally in the context of free probability (see
\cite{BelinschiNicaBBPforKTuples}). 

Our motivation to study this problem comes from the recent developments in
\cite{FevrierNicaSzpojankowski:2024} where a surprising connection between
$c$-freeness and so called free denoising was established, which we briefly
describe below.
Let  $(\alg{M},\psi)$ be a tracial $W^*$-probability space, i.e.,
$\alg{M}$ is a von Neumann algebra with a faithful, normal trace
$\psi$.
Assume further that two random variables $a,b$ are free with respect to
$\psi$. We interpret the random variable $a$ as a ``signal'' of interest and
$b$ as a ``noise'', and we view $P(a,b)$ as a noisy version of $a$, where
$P\in\IC\langle X,Y \rangle$. We are interested in  the conditional expectation
$\E[a|P(a,b)]$, which is the best approximation of
$a$ by a function of $P(a,b)$ with respect to the $L^2$-norm with respect to $\psi$.
It was shown in \cite{FevrierNicaSzpojankowski:2024} that, under the assumptions
$a\geq 0$ and $\psi(a)=1$, if we define a new state $\varphi(c)=\psi(a c)$ then
$a,b$ are $c$-free with respect to the pair $(\varphi,\psi)$, and moreover we
have $\E[a|P(a,b)]=h(P(a,b))$ where
$h=d\mu_{P(a,b)}^\varphi/d\mu_{P(a,b)}^\psi$ is the Radon-Nikodym
derivative. Methods of determining distributions of polynomials in free
variables are known (see
\cite{BelinschiMaiSpeicherSubordination,LehnerSzpojankowski:2023:freeint1}),
and thus we fill this gap and provide a method to compute the Radon-Nikodym
derivative above, and consequently the conditional expectation.

In order to illustrate our results we present one example, which is fully
discussed in Section~\ref{sec:denoising}. Suppose $a$ is a standard
semicircular element and $b$ has distribution
$\tfrac{1}{2}\left(\delta_{-1}+\delta_1\right)$.
Let $c = i[a,b]$ be their commutator, then
for $s<1/2$ we have
\[
  E\bigl[(1-sa)^{-1} | c\bigr]=h(c)
\]
with 
\[
  h(t)=\frac{2(2+\tilde{\eta}_a^2(s))}{(1-s\tilde{\eta}_a^2(s))\left((2+\tilde{\eta}_a(s))^2-\tilde{\eta}_a^2(s)
      t^2\right)},
\]
where $\tilde{\eta}_a(s)=\frac{1-\sqrt{1-4 s^2}}{2 s}$ is the shifted Boolean transform of the standard semicircle distribution.

Expanding the above power series at $s=0$ we obtain
\begin{align*}
	\E[a^2|c]&=\frac{c^2+2}{4},\\
	\E[a^4|c]&=\frac{c^4+6c^2+12}{16}.
\end{align*}
The same formulas hold more generally for the anticommutator and in fact 
any element of the form $c=\theta ab+\bar{\theta}ba$, with $\theta\in\IC$
arbitrary such that $\abs{\theta}=1$.

In the framework of general $c$-free random variables in a two-state
$W^*$-noncommutative probability space $(\alg{M},\varphi,\psi)$ it is natural
to ask whether conditional expectations onto $c$-free subalgebras
with respect to $\varphi$ exist. This
question is non-trivial because in general $c$-freeness  is incompatible with
 a tracial state.
Indeed the examples in Section~\ref{sec:6} below show that in the $c$-free case
conditional expectations in the strict sense do not exist,
however the projection in $L^2(\alg{M})$ with respect to the inner product
\begin{equation*}
  \langle a,b\rangle  = \varphi(ab^*)
\end{equation*}
exists and serves as a formal analog of the conditional expectation.
We call this projection a \emph{right quasi-conditional expectation}, as it has only the right module property and is not positive. We denote it by $\rE_b^{\varphi}$, when we compute the projection on subalgebra generated by $b$. We show that this projection shares some properties with the conditional expectations with respect to $\varphi$. In particular for any polynomial $P$ in $c$-free variables $\rE_b^\varphi[P(a,b)]$ is a polynomial in $b$ of degree at most $\mathrm{deg}(P)$. Moreover we provide recursive formulas for $\rE_b[P(a,b)]$ similar to those obtained for conditional expectations for polynomials in free variables from \cite{LehnerSzpojankowski:2023:freeint1}. Using this machinery, we show that the right quasi-conditional expectation of the resolvent of a sum of $c$-free variables is again a resolvent, multiplied by an analytic function. This resembles the phenomenon of subordination for free additive convolution from \cite{Biane98}. In Section~\ref{sec:sigmatransform} we show that using the right quasi-conditional expectation one can derive a formula for the $\Sigma$-transform for $c$-free multiplicative convolution from \cite{PopaWang:2011:multiplicative}.

This paper is organized as follows. In Section~\ref{sec:2} we recall basic definitions and facts concerning freeness and $c$-freeness. In Section~\ref{sec:VNRP} we prove a characterization of $c$-freeness in terms of Boolean cumulants. In Section~\ref{sec:4} we provide a method of computing distributions of polynomials in $c$-free variables. In Section~\ref{sec:denoising} we present an application of our result to a free denoising problem. In Section~\ref{sec:6} we study right quasi-conditional expectations and present a recursion for them based on Boolean cumulants. In Section~\ref{sec:sigmatransform} we derive the $\Sigma$-transform for $c$-free multiplicative convolution using  right quasi-conditional expectations.

\section{Preliminaries}\label{sec:2}
\subsection{Noncommutative probability spaces}
A \emph{noncommutative probability space} is a pair
 $(\alg{M},\psi)$, where $\alg{M}$ is a unital $*$-algebra and
$\psi:\alg{M}\mapsto\IC$ is a positive unital linear functional, commonly
called a \emph{state}. When $\alg{M}$ is a von Neumann algebra and $\psi$ is a
faithful normal state, the pair  $(\alg{M},\psi)$ is  called  a \emph{$W^*$-probability space}.

A \emph{two-state noncommutative probability space} is a triple
$(\alg{M},\varphi,\psi)$ where $\alg{M}$ is equipped with
two states $\varphi$ and $\psi$.
We will assume throughout that the second state $\psi$ is tracial.

\begin{definition}\label{def:2.2}
  \begin{enumerate}[label=(\roman*)]
   \item []
   \item 
  A family of subalgebras $(\alg{A}_i)_{i\in I}$ of a noncommutative probability space $(\alg{M},\psi)$ is called \emph{free} or \emph{free independent} if
  \begin{equation*}
    \psi(u_1u_2\dotsm u_n)=  0
  \end{equation*}
  for any choice of $u_j\in \bigcup_i\alg{A}_i$ such that $\psi(u_j)=0$ and
  $u_j\in\alg{A}_{i_j}$ with $i_j\ne i_{j+1}$ for all $j\in\{1,2,\dots,n-1\}$.
   \item 
  A family of subalgebras $(\alg{A}_i)_{i\in I}$ of a two-state noncommutative probability space $(\alg{M},\varphi,\psi)$ is called
  \emph{conditionally free} or \emph{$c$-free} if
  \begin{equation*}
   \psi(u_1u_2\cdots u_n) = 0\qquad\mbox{and}\qquad \varphi(u_1u_2\dotsm u_n)=  \varphi(u_1)\varphi(u_2)\dotsm\varphi(u_n)
  \end{equation*}
  whenever $u_j\in \bigcup_i\alg{A}_i$ with $\psi(u_j)=0$ and
  $u_j\in\alg{A}_{i_j}$ with $i_j\ne i_{j+1}$ for all $j\in\{1,2,\dots,n-1\}$,
  i.e., the family is free with respect to $\psi$ and in addition satisfies the
  factorization property with respect to $\varphi$.
  \end{enumerate}
\end{definition}

\subsection{Cumulants}
To quote Fisher  \cite{Fisher:1929:moments},
 when it comes to the description of joint distributions of random variables,
\emph{the formulae are
much simplified by the use of cumulative
moment functions, or semi-invariants, in
place of the crude moments}.
This is even more true for non-commuting ones. In the present paper we will deal with three kinds of cumulants.

\begin{notation}
  \begin{enumerate}[label=(\roman*), leftmargin=2em]
   \item []
   \item We set $[n]\coloneqq\{1,\ldots,n\}$.
    A \emph{set partition of order $n$}
    is a collection     $\pi=\{V_1,\ldots,V_m\}$
    of non-empty pairwise disjoint subsets 
    $V_i\subseteq [n]$, called \emph{blocks},
    such that $\bigcup_{i=1}^m V_i = [n]$.
    Every set partition $\pi$ of order $n$ induces an equivalence relation
    on $[n]$, denoted by $\sim_\pi$,
    such that  $k\sim_\pi    l$ if and only if $k$ and $l$ are elements of the
    same block of $\pi$.
   \item Conversely, any function $f:[n]\to C$  into some set of ``colours''
    $C$ induces an equivalence relation $i\sim j:\iff{} f(i)=f(j)$.
    The corresponding set partition $\pi=\ker f$ is called
    the \emph{kernel} of $f$.
   \item We say that a set partition $\pi$ of $[n]$ is
    \emph{noncrossing} if there is no pair of  blocks $V,W\in \pi$ with
    elements $a,c\in V$ and $b,d\in W$ such that $a<b<c<d$. The collection of
    noncrossing partitions of $[n]$ is denoted $\NC(n)$.  
   \item An \emph{interval partition} of $[n]$ is a noncrossing partition
    $\pi\in \NC(n)$ such that every block is an  interval of $[n]$. The
    collection of interval partitions of $[n]$ is denoted by $\IntPart(n)$. 
   \item Let $\pi\in \NC(n)$ and $V,W\in \pi$ be two different blocks.
    We say  that a block
$W$ is \emph{nested} inside a block $V$
    if $\min(V)\leq w\leq  \max(V)$ for any $w\in
    W$. A block $W\in \pi$ is called an \emph{outer block} if there is no $V\in
    \pi$ such that $W$ is nested inside $V$. A block $W$ that is not outer is called an
    \emph{inner block}.  We denote the sets of inner and outer blocks of
    $\pi\in \NC(n)$ by $\IB(\pi)$ and $\OB(\pi)$, respectively. An interval
    partition contains no inner blocks. 
\item  For a family of multilinear functionals
    $(L_n:\alg{M}^n\to\mathbb{C})_{n\in\IN}$ and a partition $\pi$ we denote by  
    $L_\pi$ the partitioned multilinear functional
    $$
    L_\pi(a_1,a_2,\dots,a_n) = \prod_{V\in\pi} L_{V}( a_1,a_2,\ldots,a_n),
    $$
    where for a block $V = \{i_1<i_2<\cdots <i_s\}$
    we denote $ L_{V}( a_1,a_2,\ldots,a_n)    := L_s(a_{i_1},a_{i_2},\ldots,a_{i_s}).$
    This notation is extended to linear functionals by setting
     $\varphi_n(b_1,b_2,\ldots,b_n):=\varphi(b_1b_2\cdots b_n).$ 
   \item Both $\NC(n)$ and $\IntPart(n)$ are lattices with respect to
    the \emph{reversed refinement order}:
    two partitions $\pi$ and $\rho$ of $[n]$ satisfy the relation $\pi\leq\rho$ 
    for every block $V\in\pi$ there is a block $W\in\rho$ such that $V\subseteq W$. 
  \end{enumerate}
\end{notation}

\begin{definition}\label{def:cumulants}
  \begin{enumerate}[label=(\roman*)]
   \item []
   \item 
  Let $(\mathcal{M},\psi)$ be a noncommutative probability space. The following identities uniquely determine families of multilinear
  functionals $(\beta^\psi_n)_{n\in\IN}$,
  $(r^\psi_n)_{n\in\IN}$,
  called \emph{Boolean} and \emph{free} cumulants:
\begin{align}
\psi(a_1a_2\dotsm a_n)&=\sum_{\sigma \in \IntPart(n)}\beta^\psi_{\sigma}(a_1,a_2,\ldots,a_n),\label{eq:mcboolean}\\
\psi(a_1a_2\dotsm a_n)&=\sum_{\sigma \in  \NC(n)}r^\psi_{\sigma}(a_1,a_2,\ldots,a_n).\label{eq:mcfree}
\end{align}
\item 
	Let $(\mathcal{M},\varphi,\psi)$ be a two-state noncommutative probability space. The following identity uniquely determines the family of \emph{$c$-free cumulants}: $(r_{n}^{\varphi,\psi})_{n\in\mathbb{N}}$
\begin{align}
  \varphi(a_1a_2\cdots a_n)
  &=\sum_{\sigma\in \NC(n)}
    \prod_{
V\in\OB(\sigma)              
  }
  r_{V}^{\varphi,\psi}(a_1,\ldots,a_n)
  \prod_{
W\in\IB(\sigma)  
  }
  r_{W}^{\psi}(a_1,\ldots,a_n)
  .\label{eq:mccfree}
	\end{align}

  \end{enumerate}
\end{definition}
\begin{remark}
	Note that when dealing with a two-state noncommutative probability space $(\mathcal{M},\varphi,\psi)$ we will use Boolean cumulants $\beta^{\varphi}$ and $\beta^\psi$ with respect to both states $\varphi$ and $\psi$.
\end{remark}

\begin{proposition}
  \label{prop:vanishingmixedcumulants}
  \begin{enumerate}[label=(\roman*)]
   \item []
   \item Subalgebras are free if and only if mixed free cumulants vanish \cite{SpeicherNC}.
   \item Subalgebras are $c$-free if and only if mixed free cumulants and mixed
    $c$-free cumulants vanish \cite{BozejkoLeinertSpeicher:1996:convolution}.

  \end{enumerate}
\end{proposition}

A similar result holds for Boolean independence which however we do not
consider in this paper. Although one might conclude from this characterization
that every notion of independence is best studied in terms of its ``native''
variant of cumulants,  it turned out recently that Boolean cumulants are the
most universal ones and useful for different kinds of independences
\cite{JekelLiu:2019:operad}, 
in particular
free probability \cite{FMNS2,LehnerSzpojan,LehnerSzpojankowski:2023:freeint1},
but there are also applications in classical probability
\cite{SaulisStatulevicius:1991:limit,VanWerdeSanders:2023:matrix}.
In the present paper we show that they are  useful for
the study of $c$-freeness as well.

The following recursive reformulations of the Boolean moment-cumulant relations
turn out to be most useful for our study:
\begin{align}
  \label{eq:recurrenceboolcum}
  \varphi(a_1a_2\dotsm a_n)
  &= \sum_{k=1}^n
    \beta_k^\varphi(a_1,a_2,\dots,a_k)\,\varphi(a_{k+1}a_{k+2}\dotsm a_n),
    \\
  \label{eq:recurrenceboolcumpsi}
  \psi(a_1a_2\dotsm a_n)
  &= \sum_{k=1}^n
    \beta_k^\psi(a_1,a_2,\dots,a_k)\,\psi(a_{k+1}a_{k+2}\dotsm a_n).
\end{align}

\begin{remark}
The previous recursion has a natural translation into the language of
generating functions. More precisely, let $(\alg{M},\psi)$ be a noncommutative
probability space and let $a\in \alg{M}$. Then, the (ordinary) \emph{moment
  generating function} (or \emph{moment transform}) of a random variable $a$ with respect to $\psi$ is defined by
\[M_a^\psi(z) \coloneqq 1 + \sum_{n=1}^\infty \psi(a^n)z^n.\]
Similarly, we define the \emph{Boolean transform} and the \emph{shifted Boolean
  transform of a random variable $a$} as the generating function of its cumulants
\[
\eta_a^\psi(z) \coloneqq  \sum_{n=1}^\infty \beta_n^\psi(a)z^n,\qquad \tilde\eta^\psi_a(z) \coloneqq  \frac{1}{z}\eta_a^\psi(z) = \sum_{n=1}^\infty \beta_n^\psi(a)z^{n-1},
\]
respectively, where $\beta_n^\psi(a) := \beta_n^\psi(a,a,\ldots,a).$
It is not difficult to verify that the recurrence
\eqref{eq:recurrenceboolcumpsi} is equivalent to the functional equation
\[
M_a^\psi(z) = \frac{1}{1-\eta_a^\psi(z)}.
\]
\end{remark}

\subsection{Closure operators and Möbius inversion}
The following simple lemma subsumes most combinatorial proofs
of the present paper.
\begin{definition}
  Let $P$ be a poset. A map $c:P\to P$ is called \emph{closure operator}
 (or   \emph{sample operator}  \cite {BarnabeiBriniRota:1986:moebius})
  if the following conditions are satisfied:
  \begin{enumerate}[label=(\roman*)]
   \item it is \emph{increasing}, i.e., $x\leq c(x)$ for every $x\in P$; 
   \item it is \emph{order preserving}, i.e., if $x\leq y$ then $c(x)\leq c(y)$; 
   \item it is \emph{idempotent}, i.e., $c\circ c=c$.
  \end{enumerate}
  An element $x\in P$ is called \emph{closed} if $x=c(x)$.
  The set of closed elements is denoted by $\overline{P}$
  and inherits the order.
\end{definition}

\begin{lemma}\label{lem:closure}
  Let $(P,\leq)$ be a poset and $c:P\to P$ a closure operator
  and $(\overline{P},\leq)$ the subposet of closed elements with the induced order.
  Given a function $f:P\to \IC$,  define its partial sums
  $$
  F(x) = \sum_{\substack{y\in P\\ y\leq x}} f(y)
  .
  $$
  Assume that there is a function $g:\overline{P}\to \IC$
  such that for every $x\in\overline{P}$ we also have
  $F(x) = \sum_{ 
    \substack{y\in \overline{P}\\ y\leq x}} g(y)$.
  Then
  $$
  g(y) = \sum_{\substack{z\in P\\ c(z)=y}} f(z)
  $$
   for all $y\in\overline{P}$.
\end{lemma}

\begin{proof}
   For every $x\in \overline{P}$ we can write
  $$
  F(x) 
  = \sum_{z\leq x} f(z)
  = \sum_{\substack{y\in\overline{P}\\ y\leq x}} \sum_{\substack{z\in P\\ c(z)=y}}f(z)
  $$
  and apply M\"obius inversion on $\overline{P}$.
    
  \end{proof}

\subsection{Cumulants in terms of cumulants}
\begin{definition}
  \label{def:intervalclosure}
  The \emph{interval closure} $\overline{\pi}$ of a noncrossing partition $\pi$
  is the smallest interval partition $\rho$ such that $\pi\leq\rho$.
  Its blocks are  the convex hulls of the outer blocks of $\pi$.
  A noncrossing partition $\pi\in \NC(n)$ is called \emph{irreducible} if
  $\overline{\pi}=\hat{1}_n$, i.e., it has a unique outer block or equivalently,
  $1\sim_\pi n$.
  The restrictions of $\pi$ to the blocks of $\overline{\pi}$
  are called its \emph{irreducible components} and
  it follows that $\pi$ is the concatenation of its irreducible components.
  We denote the set of irreducible noncrossing partitions by
  $\NCirr(n)$.
\end{definition}

Irreducible partitions are a natural means to express mutual relations between
different kinds of cumulants
\cite {Lehner:2002:connected,ArizmendiHasebeLehnerVargas:2015:relations},
in particular:

\begin{proposition}
  \begin{enumerate}[label=(\roman*)]
   \item []
    
   \item 
    Let $(\alg{M},\psi)$ be a noncommutative probability space and
    $(r^\psi_n)_{n\geq1},(\beta_n^\psi)_{n\geq1}$ be the sequences of free and
    Boolean cumulant functionals, respectively. Then
    \begin{equation}
      \label{eq:beta=sumNCirrpi}
      \beta_n^{\psi}(a_1,a_2,\dots,a_n) = \sum_{\pi\in\NCirr(n)}    r_\pi^{\psi}(a_1,a_2,\dots,a_n)
    \end{equation}
    for any $n\geq1$ and $a_1,\ldots,a_n\in \alg{M}$.
   \item 
    Let $(\alg{M},\varphi,\psi)$ be a noncommutative probability space,
then for any $n\geq1$ and $a_1,\ldots,a_n\in \alg{M}$
    the Boolean cumulants can be expressed in terms of the $c$-free and free
    cumulants as follows:
    \begin{equation}
      \label{eq:betaphi=sumNCirrpi}
      \beta_n^{\varphi}(a_1,a_2,\dots,a_n) = \sum_{\pi\in\NCirr(n)}
      \prod_{V\in \OB(\pi)}
      r_V^{\varphi,\psi}(a_1,a_2,\dots,a_n)
      \prod_{W\in \IB(\pi)}
      r_W^{\psi}(a_1,a_2,\dots,a_n)
      .
    \end{equation}
    \begin{proof}
      We note that \eqref{eq:betaphi=sumNCirrpi} does not
      appear explicitly in the literature,
      but as \eqref{eq:beta=sumNCirrpi}
      it follows
      from Lemma~\ref{lem:closure}
      applied to the poset $\NC(n)$ with
      the interval closure operator
      from Definition~\ref{def:intervalclosure},
      where for $\pi\in\NC(n)$, we set
      \begin{align*}
        F(\pi)&=\varphi_\pi(a_1,\ldots,a_n) , \\
        f(\pi)&=\prod_{V\in \OB(\pi)} 
        r_V^{\varphi,\psi}(a_1,a_2,\dots,a_n)
        \prod_{W\in \IB(\pi)}
                r_W^{\psi}(a_1,a_2,\dots,a_n) ,\\
        \intertext{and for $\pi\in\IntPart(n)$ we set}
        g(\pi)&=\beta_{\pi}(a_1,a_2,\dots,a_n)
                .
      \end{align*}

    \end{proof}
  \end{enumerate}
	
\end{proposition}

\begin{remdef}
  \label{remdef:fCAC}
  \begin{enumerate}[label=(\roman*)]
   \item []
    
   \item
    Let $(\alg{M},\psi)$ be a noncommutative probability space.
    Assume that $\alg{A}_1,\ldots,\alg{A}_s$ are free subalgebras of $\alg{M}$.
    As an immediate consequence of the vanishing of mixed cumulants
    (Proposition~\ref{prop:vanishingmixedcumulants}) and the preceding
    proposition, Boolean cumulants of free random variables satisfy the
    property of \emph{vanishing of Cyclically Alternating Cumulants},
    in short property (CAC):
    \begin{align*}
      \beta_n^\psi(a_1,a_2,\dots,a_n) &= 0
    \end{align*}
    whenever $a_1,a_2,\dots,a_n\in\bigcup_{i=1}^s\alg{A}_i$ and
    $a_1$ and $a_n$ come from different subalgebras.
   \item
    Similarly, in a two-state noncommutative probability space
    $(\alg{M},\varphi,\psi)$ 
    space  property (CAC) is satisfied
    by the Boolean cumulants with respect to both states:
    Assume
    that $\alg{A}_1,\ldots,\alg{A}_s$ are $c$-free subalgebras of $\alg{M}$. 
    Then
    \begin{align*}
      \beta_n^\psi(a_1,a_2,\dots,a_n) &= 0\\
      \beta_n^\varphi(a_1,a_2,\dots,a_n) &= 0
    \end{align*}
    whenever $a_1,a_2,\dots,a_n\in\bigcup_{i=1}^s\alg{A}_i$ and
    $a_1$ and $a_n$ come from different subalgebras.
  \end{enumerate}
  
\end{remdef}

\begin{definition}
  \label{def:irrreford}
  The  \emph{irreducible refinement order} on $\NC(n)$ is defined as follows: Let $\pi,\sigma\in \NC(n)$. 
  We say that $\pi\ll\sigma$ if $\pi\leq\sigma$ and in addition the restrictions of $\pi$ to the blocks of
  $\sigma$ are irreducible, i.e., for every $W\in\sigma$ we have $\min(W)\sim_\pi \max(W)$.
\end{definition}

The irreducible refinement order was introduced in
\cite {BelinschiNicaBBPforKTuples}
and \cite {PetrulloSenato:2011:kerov} 
(see also \cite {Nica:2010:linked}, \cite {BianeJosuatVerges:2019:bruhat})
in order to generalize formula \eqref{eq:beta=sumNCirrpi} to partitioned
cumulants.

\begin{proposition}
	Let $(\alg{M},\psi)$ be a noncommutative probability space and
        $(r^\psi_n)_{n\geq1},(\beta_n^\psi)_{n\geq1}$ be the sequences of free
        and Boolean cumulant functionals, respectively. Then
	\begin{equation}
		\label{eq:betarho=sumNCirrpi}
		\beta_\rho^\psi(a_1,a_2,\dots,a_n) = \sum_{\pi\ll\rho}    r_\pi^\psi(a_1,a_2,\dots,a_n)
	\end{equation}
	for any $n\geq1$, $\rho\in \NC(n)$ and $a_1,\ldots,a_n\in \alg{M}$.
\end{proposition}

\subsection{Recurrence diagrams}
Let $(\alg{M},\varphi,\psi)$ be a two-state noncommutative probability space. We take the basic pattern of the recurrences from
\cite[p.366]{BozejkoLeinertSpeicher:1996:convolution}.
The free and $c$-free cumulants satisfy the so-called \emph{unshuffle}  recurrence
$$
\begin{tikzpicture}[anchor=base,baseline]
	\paczek{0.0,0.0}{\psi}
\end{tikzpicture}
=
\sum
\begin{tikzpicture}[anchor=base,baseline]
  \node at (-0.75,0.8) {$r^{\psi}$};
  \draw (-0.75,-0.4)--(-0.75,0.75)--(5.25,0.75)--(5.25,-0.4);
  \draw (0.75,-0.4)--(0.75,0.75);
  \draw (2.25,-0.4)--(2.25,0.75);
  \draw (3.75,-0.4)--(3.75,0.75);
  
  \paczek{0.0,0.0}{\psi}
  \paczek{1.5,0.0}{\psi}
  \node at (3.0,0.0){$\cdots$};
  \paczek{4.5,0.0}{\psi}
  \paczek{6.0,0.0}{\psi} 
\end{tikzpicture}
$$
$$
\begin{tikzpicture}[anchor=base,baseline]
  \paczek[lightgray]{0.0,0.0}{\varphi}
\end{tikzpicture}
=
\sum
\begin{tikzpicture}[anchor=base,baseline]
  \node at (-0.75,0.8) {$r^{\varphi,\psi}$};
  \draw (-0.75,-0.4)--(-0.75,0.75)--(5.25,0.75)--(5.25,-0.4);
  \draw (0.75,-0.4)--(0.75,0.75);
  \draw (2.25,-0.4)--(2.25,0.75);
  \draw (3.75,-0.4)--(3.75,0.75);
  
  \paczek{0.0,0.0}{\psi}
  \paczek{1.5,0.0}{\psi}
  \node at (3.0,0.0){$\cdots$};
  \paczek{4.5,0.0}{\psi}
  \paczek[lightgray]{6.0,0.0}{\varphi} 
\end{tikzpicture}
$$
while the respective Boolean cumulants satisfy the so-called
\emph{deconcatenation} recurrence  \eqref{eq:recurrenceboolcum} and
\eqref{eq:recurrenceboolcumpsi}: 
\begin{equation}\label{fig:boolRecPsi}
  \begin{tikzpicture}[anchor=base,baseline]
    \paczek{0.0,0.0}{\psi}
  \end{tikzpicture}
  =
  \sum
\begin{tikzpicture}[anchor=base,baseline]
    \node at (-0.25,0.45) {$\beta^{\psi}$};
    \draw (-0.25,-0.25)--(-0.25,0.4)--(3.25,0.4)--(3.25,-0.25);
    \draw (0.25,-0.25)--(0.25,0.4);
    \draw (0.75,-0.25)--(0.75,0.4);
    \draw (1.25,-0.25)--(1.25,0.4);
\node at (2.0,-0.1){$\cdots$};
\draw (2.75,-0.25)--(2.75,0.4);  
    \paczek{4.0,0.0}{\psi} 
  \end{tikzpicture}
\end{equation}
			
$$
\begin{tikzpicture}[anchor=base,baseline]
  \paczek[lightgray]{0.0,0.0}{\varphi}
\end{tikzpicture}
=
\sum
\begin{tikzpicture}[anchor=base,baseline]
  \node at (-0.25,0.45) {$\beta^{\varphi}$};
  \draw (-0.25,-0.25)--(-0.25,0.4)--(3.25,0.4)--(3.25,-0.25);
  \draw (0.25,-0.25)--(0.25,0.4);
  \draw (0.75,-0.25)--(0.75,0.4);
  \draw (1.25,-0.25)--(1.25,0.4);
\node at (2.0,-0.1){$\cdots$};
\draw (2.75,-0.25)--(2.75,0.4);  
  \paczek[lightgray]{4.0,0.0}{\varphi} 
\end{tikzpicture}
$$
and consequently
the relations between noncrossing and Boolean cumulants follow a similar pattern:
$$
\begin{tikzpicture}[anchor=base,baseline]
  \paczek{0.0,0.0}{\beta^\psi}
\end{tikzpicture}
=
\sum
\begin{tikzpicture}[anchor=base,baseline]
  \node at (-0.75,0.8) {$r^{\psi}$};
  \draw (-0.75,-0.4)--(-0.75,0.75)--(5.25,0.75)--(5.25,-0.4);
  \draw (0.75,-0.4)--(0.75,0.75);
  \draw (2.25,-0.4)--(2.25,0.75);
  \draw (3.75,-0.4)--(3.75,0.75);
  
  \paczek{0.0,0.0}{\psi}
  \paczek{1.5,0.0}{\psi}
  \node at (3.0,0.0){$\cdots$};
  \paczek{4.5,0.0}{\psi}
\end{tikzpicture}
$$
$$
\begin{tikzpicture}[anchor=base,baseline]
  \paczek[lightgray]{0.0,0.0}{\beta^\varphi}
\end{tikzpicture}
=
\sum
\begin{tikzpicture}[anchor=base,baseline]
  \node at (-0.75,0.8) {$r^{\varphi,\psi}$};
  \draw (-0.75,-0.4)--(-0.75,0.75)--(5.25,0.75)--(5.25,-0.4);
  \draw (0.75,-0.4)--(0.75,0.75);
  \draw (2.25,-0.4)--(2.25,0.75);
  \draw (3.75,-0.4)--(3.75,0.75);
  
  \paczek{0.0,0.0}{\psi}
  \paczek{1.5,0.0}{\psi}
  \node at (3.0,0.0){$\cdots$};
  \paczek{4.5,0.0}{\psi}
\end{tikzpicture}.
$$
The unshuffle and deconcatenation recurrences find a conceptual explanation within the Hopf-algebraic approach to moments and different kinds of cumulants developed in \cite{EbrahimiFardPatras2020}.
\subsection{Cumulants with products as entries}
The classical formula of Leonov-Shiryaev for cumulants with products as entries
has its natural analogue for Boolean cumulants.
It is again a consequence of Lemma~\ref{lem:closure},
this time the closure operator being $c(\pi)=\pi\vee\rho$.
\begin{proposition}
  \label{prop:boolprod}
  Let $a_1,a_2,\dots,a_n\in \alg{A}$ be random variables
  then
  \begin{equation}
    \label{eq:BoolProd}
    \beta_{m+1}(a_1a_2\cdots a_{d_1},a_{d_1+1}a_{d_1+2}\cdots a_{d_2},\ldots,a_{d_{m}+1}a_{d_{m}+2}\cdots a_{n})
    =\sum_{\substack{\pi \in \IntPart(n)\\ \pi \vee \rho=1_n}}\beta_\pi(a_1,a_2,\ldots,a_n),
  \end{equation}
  where $\rho=\{\{1,2,\ldots,d_1\},\{d_1+1,d_1+2,\ldots,d_2\},\ldots,\{d_m+1,\ldots,n\}\}\in \IntPart(n)$, and $\vee$ is the join in the lattice of interval partitions.
  The condition $\pi \vee \rho=1_n$ is equivalent to
  \begin{equation*}
\pi \geq \{\{1\},\{2\},\dots,\{d_1-1\},\{d_1,d_1+1\},\{d_1+2\},\dots,\{d_m-1\},\{d_m,d_m+1\},\dots,\{d_{n}\}\}.
  \end{equation*}
\end{proposition}

The expansion \eqref{eq:BoolProd} has the following recursive reformulation
\cite{LehnerSzpojankowski:2023:freeint1}.
\begin{corollary}
  \label{cor:RecursiveBoolProd}
  Let $a_1,a_2,\dots,a_n\in \alg{A}$ be random variables.
  Given an  interval partition
  $\rho=\{\{1,\ldots,d_1\},\{d_1+1,\ldots,d_2\},\ldots,\{d_{m-1}+1,\ldots,n\}\}\in
  \IntPart(n)$, the Boolean cumulant with the corresponding products as entries
  is
      \begin{equation}
    \label{eq:RecursiveBoolProd}
    \begin{multlined}
      \beta_{m}(a_1a_2\cdots a_{d_1},a_{d_1+1}a_{d_1+2}\cdots a_{d_2},\ldots,a_{d_{m-1}+1}a_{d_{m-1}+2}\cdots a_{n})\\
      = \sum_{k=1}^n \sum_{d_{k-1}<j<d_k} 
         \beta_{j}(a_1,a_2,\ldots,a_j)\beta_{m-k+1}(a_{j+1}a_{j+2}\cdots
         a_{d_k},a_{d_k+1}a_{d_k+2}\dotsm a_{d_{k+1}},\ldots,a_{d_{m-1}+1}\cdots a_n).
    \end{multlined}
  \end{equation}
\end{corollary}

\section{VNRP for $c$-freeness}
\label{sec:VNRP}

In this section, we derive a formula for the Boolean cumulants of $c$-free
elements, generalizing the results from \cite{FMNS2} which express Boolean
cumulants of free elements using noncrossing partitions with so-called
\emph{vertical no-repeat property}.  
\par In order to state the next combinatorial result, we fix $m \in \mathbb{N}$
and a colouring $c:\{1,\ldots,m\}\to\{1,\ldots,s\}$.
Then, we define 
$$
\NC(m;c) := \{\sigma\in \NC(m)\;|\; \sigma\leq \ker c\},
$$
the set of noncrossing partitions compatible with $c$. We also set $\NCirr(m;c):=\NC(m;c)\cap\NCirr(m)$.

\begin{theorem}[{\cite[Theorem 1.1]{FMNS2}}]
  \label{thm:vnrp}
  \noindent
  For every $\sigma \in \NC(m; c)$ there exists a unique maximal $\rho \in \NC(m;
  c)$ such that $\sigma \ll \rho$, which is characterized by the
  \emph{Vertical-No-Repeat property (VNRP)}, that is,
  every inner block of $\rho$ is nested immediately inside a block of a different colour.
\end{theorem}
Our goal is to extend the characterization of freeness from \cite{FMNS2} to the setting of a two-state
noncommutative probability space $(\mathcal{M},\varphi,\psi)$ and give a similar characterization of conditional freeness of Bożejko, Leinert and Speicher
\cite{BozejkoLeinertSpeicher:1996:convolution}.

The following ``$c$-free Boolean cumulant functionals''
$\beta^{\varphi,\psi}$,
depending on both states like the $c$-free cumulants \eqref{eq:mccfree},
appear naturally in the generalization of \eqref{eq:betaphi=sumNCirrpi}.
\begin{definition}
	\label{def:BooleanChiPhi}
	Let $(\alg{M},\varphi,\psi)$ be a two-state noncommutative probability space, and fix $n \in \mathbb{N}$ and $\pi \in \NC(n)$.  
	We define the \emph{nested two-state Boolean cumulant functional} $\beta^{\varphi,\psi}_\pi:\mathcal{M}^n\to\mathbb{C}$ by the formula
	\begin{align*}
		\beta^{\varphi,\psi}_\pi(a_1,\ldots,a_n)
		:=\prod_{
V\in\OB(\pi)            
		}
		\beta_{V}^{\varphi}\bigl(a_1,\ldots,a_n\bigr)
		\prod_{
W\in\IB(\pi)  
		}
		\beta_{W}^{\psi}\bigl(a_1,\ldots,a_n\bigr),
	\end{align*}
	for every $a_1,\ldots,a_n\in \alg{M}$.
\end{definition}
\begin{remark}
	Observe that for an interval partition $\pi\in \IntPart(n)$ we have 
	\[\beta^{\varphi,\psi}_\pi(a_1,\ldots,a_n)=\beta^{\varphi}_\pi(a_1,\ldots,a_n).\]
\end{remark}

\begin{proposition} \label{prop:35}
	Let $(\alg{M},\varphi,\psi)$ be a two-state noncommutative probability space,
	fix $n \in \mathbb{N}$ and $\pi \in \NC(n)$. Then for every
	$a_1,\ldots,a_n\in\alg{M}$ the nested two-state Boolean cumulants are given by
	\[\beta_\pi^{\varphi,\psi}(a_1,\ldots,a_n)=\sum_{\rho\ll \pi}
	\prod_{
V\in\OB(\rho)                   
	}
	r_{V}^{\varphi,\psi}\bigl(a_1,\ldots,a_n\bigr)
	\prod_{
W\in\IB(\rho)
	}
	r_{W}^{\psi}\bigl(a_1,\ldots,a_n\bigr).\]
\end{proposition}
\begin{proof}
	We know that
	\begin{align*}
		\beta^\varphi_n(a_1,\ldots,a_n)
		=& \sum_{\substack{
				\rho\ll 1_n   \\
				1\in V_0
		} }
		r_{V_0}^{\varphi,\psi}\bigl(a_1,\ldots,a_n\bigr)
		\prod_{\substack{
				V\in\rho  \\
				V\neq V_0
		}}
		r_{V}^{\psi}\bigl(a_1,\ldots,a_n\bigr),
		\\
		\beta_n^{\psi}(a_1,a_2,\dots,a_n) = & \sum_{\pi\in\NCirr(n)}    r_\pi^{\psi}(a_1,a_2,\dots,a_n)
	\end{align*}
	Plugging both formulas into the formula from Definition~\ref{def:BooleanChiPhi} we immediately obtain the claim.
\end{proof}

The above proposition is crucial in the proof of the characterization of $c$-freeness in terms of Boolean cumulants below.

\begin{theorem}
  \label{thm:VNTPCfree}
  Let $( \mathcal{M} , \varphi,\psi )$ be a two-state noncommutative
  probability space. Assume that  unital subalgebras
  $\mathcal{A}_1,\ldots,\mathcal{A}_s\subseteq \mathcal{M}$
  are free with respect to $\psi$. 
  Then the following two assertions are equivalent.
  \begin{enumerate}[label=(\roman*)]
   \item
    $\mathcal{A}_1,\ldots,\mathcal{A}_s$ are $c$-free
    independent with respect to $(\varphi,\psi)$.
    
   \item 	
    For every $n \in \mathbb{N}$, every colouring 
    $c : \{ 1, \ldots , n \} \to \{ 1, \ldots , s \}$, and every 
    $x_1 \in \mathcal{A}_{c(1)}, \ldots, x_n \in \mathcal{A}_{c(n)}$, one has
    \begin{equation}
      \label{eq:VNRP:cfree}
      \beta_n^{\varphi} (x_1, \ldots , x_n) = 
      \sum_{
\substack{
          \pi \in \NCirr(n;c) \\
          \textnormal{with VNRP}
        }
      }
      \beta_{\pi}^{\varphi,\psi}(x_1,\ldots,x_n).
    \end{equation}
  \end{enumerate}
\end{theorem}

\begin{remark}
  \label{rmk:VNRPDiagrammatic}
  \begin{enumerate}[label=(\roman*)]
   \item []
   \item 
  The special case $\varphi=\psi$ reproduces
  the characterization of freeness from \cite{FMNS2}.
 \item 
  In the case of alternating free arguments,
  formula \eqref{eq:VNRP:cfree}
  can be represented in a diagram as follows:
  In the free case we have
  \begin{align}
    \label{eq:VNRPpsi:pic}
    \beta^\psi(a_0,b_1,a_1,\dots,b_n,a_n)
    &=
      \sum
\begin{tikzpicture}[anchor=base,baseline]
        \node at (-1,0.8) {$\beta^{\psi}$};
        \draw (-1,-0.4)--(-1,0.75)--(7,0.75)--(7,-0.4);
        \draw (1,-0.4)--(1,0.75);
        \draw (3,-0.4)--(3,0.75);
        \draw (5,-0.4)--(5,0.75);
\node at (-1,-0.8){$a_0$};
        \node at (1,-0.8){$a_{i_1}$};
        \node at (3.1,-0.8){$a_{i_2}$};
        \node at (5.3,-0.8){$a_{i_{k-1}}$};
        \node at (7,-0.8){$a_{n}$};
\paczek{0.0,0.0}{\beta^\psi}
        \paczek{2,0.0}{\beta^\psi}
        \node at (4,0.0){$\cdots$};
        \paczek{6,0.0}{\beta^\psi}
      \end{tikzpicture}
    \\
    \intertext{
    and the $c$-free generalization \eqref{eq:VNRP:cfree}
follows the same pattern:}
    \label{eq:VNRPphi:pic}
    \beta^\varphi(a_0,b_1,a_1,\dots,b_n,a_n)
    &=
      \sum
\begin{tikzpicture}[anchor=base,baseline]
        \node at (-1,0.8) {$\beta^{\varphi}$};
        \draw (-1,-0.4)--(-1,0.75)--(7,0.75)--(7,-0.4);
        \draw (1,-0.4)--(1,0.75);
        \draw (3,-0.4)--(3,0.75);
        \draw (5,-0.4)--(5,0.75);
\node at (-1,-0.8){$a_0$};
        \node at (1,-0.8){$a_{i_1}$};
        \node at (3.1,-0.8){$a_{i_2}$};
        \node at (5.3,-0.8){$a_{i_{k-1}}$};
        \node at (7,-0.8){$a_{n}$};
\paczek{0.0,0.0}{\beta^\psi}
        \paczek{2,0.0}{\beta^\psi}
        \node at (4,0.0){$\cdots$};
        \paczek{6,0.0}{\beta^\psi}
      \end{tikzpicture}
  \end{align}
  \end{enumerate}
\end{remark}

\begin{proof}[Proof of Theorem~\ref{thm:VNTPCfree}]
  Proof of $(i)\implies (ii)$.

  This is essentially another application of Lemma~\ref{lem:closure}
  where   this time the poset is $(\NC(m;c), \ll)$
  and the closure   $c(\pi)$ of an element $\pi$ is the unique VNRP
  closure of $\pi$ guaranteed by Theorem~\ref{thm:vnrp}.

  We start with an application of
  Proposition~\ref{prop:35},
  where we remember that for every interval partition $\pi$ we have $\beta^{\varphi,\psi}_\pi=\beta^{\varphi}_\pi$
  \begin{align*}
    \beta_n^{\varphi} (a_1, \ldots , a_n)
    =\sum_{\rho\in \NCirr(n)}
    \prod_{
V\in\OB(\rho)
    }
    r_{V}^{\varphi,\psi}(a_1,\ldots,a_n)
    \prod_{
W\in\IB(\rho)            
    }
    r_{W}^{\psi}(a_1,\ldots,a_n).
  \end{align*}
  Observe that from $c$-freeness, we immediately see that if $\rho$ does not
  respect the colouring then the right-hand side vanishes. Recall that by
  $\NCirr(n;c)$ we denote the set of irreducible noncrossing partitions  coloured with respect to
  $c$. Then we have: 
  \begin{align*}
    \beta_n^{\varphi} (a_1, \ldots , a_n)
    =\sum_{\substack{\rho\in \NCirr(n;c)}}
    \prod_{V\in \OB(\rho)}
    r_{V}^{\varphi,\psi}(a_1,\ldots,a_n)
    \prod_{
W\in\IB(\rho)                      
    }
    r_{W}^{\psi}(a_1,\ldots,a_n).
  \end{align*}
  Now we use Theorem~\ref{thm:vnrp}
  and observe that we can reorganize the sum
  above according to the maximal elements with respect to the
  irreducible refinement order
  $\ll$
  from Definition~\ref{def:irrreford}
  among   coloured partitions,
  and these maximal elements are exactly those with VNRP
  property: 
  \begin{align*}
    \beta_n^{\varphi} (a_1, \ldots , a_n)
    &= \sum_{
\substack{
          \pi \in \NCirr(n;c) \\
          \textnormal{with VNRP}
        }
    }
    \sum_{
\rho\ll\pi }
    r_{V_o}^{\varphi,\psi}(a_1,\ldots,a_n)
    \prod_{
V\in\IB(\rho)            
    }
    r_{V}^{\psi}(a_1,\ldots,a_n)
  \end{align*}
  where for each $\rho$ we denote by $V_o$ its unique outer block.
  We conclude by invoking again Proposition~\ref{prop:35} and we obtain 
  
  \begin{equation*}    
    \beta_n^{\varphi} (a_1, \ldots , a_n) = 
    \sum_{ \substack{
        \pi \in \NCirr(n;c)\\
        \text{with VNRP}
      }}
    \beta_{\pi}^{\varphi,\psi}(a_1,\ldots,a_n).
  \end{equation*}

The proof of implication $(ii)\implies (i)$ follows standard arguments based on the
  fact that joint cumulants determine all joint moments.  
\end{proof}

\begin{remark}
  The above gives another proof of (CAC) property from Remark and
  Definition~\ref{remdef:fCAC}, because in the above characterization if the first and last
  element are $c$-free, the sum runs over the empty set. 
\end{remark}

		\section{Distribution of polynomials in $c$-free variables} \label{sec:4}
		
In this section, we derive a method for determining the distribution of
polynomials in $ c $-free variables. Our main tool is a kind of differential calculus on
the algebra of noncommutative polynomials, and more generally on formal power
series whose coefficients are noncommutative polynomials. This framework is
essential for analyzing distributions of $ c $-free variables. In a general
noncommutative probability space, elements may satisfy nontrivial relations
(e.g., \( a^2 = 1 \))  that invalidate the Leibniz rule of the derivations defined
below. No such obstructions arise in the algebra of noncommutative
polynomials, making this setting more suitable for our purposes. 
		
\subsection{Noncommutative polynomials and formal power series}
\label{ssec:nc_polynomials}
\begin{notation}\label{notation:nc_polynomials}
  Let $\alg{X}=\{X_1,X_2,\ldots,X_n\}$ be an alphabet, i.e.~a set of letters.
  We denote by 
  \[
    \alg{X}^+=\{X_{i_1}X_{i_2}\cdots X_{i_k}\mid k\in\IN, i_j\in
    \{1,2,\dots,n\}\}
  \]
  the free semigroup generated by $\alg{X}$ and by
  $\alg{X}^*=\alg{X}^+\cup\{1\}$ the free monoid.
  Furthermore, we also denote by $\IC\langle\alg{X}\rangle= \IC\langle
  X_1,X_2,\ldots,X_n\rangle$ the free associative algebra generated by the
  variables $X_1,X_2,\ldots,X_n$, i.e., the linear span of $\alg{X}^*$ with the
  concatenation product, also known as the algebra of noncommutative
  polynomials.

  For elements $a_1,a_2,\ldots,a_n\in\alg{M}$ and
  $P\in\IC\langle X_1,X_2,\dots,X_n\rangle$ we denote by
  $P(a_1,a_2,\ldots,a_n)$ the evaluation of a polynomial
  $P\in\IC\langle X_1,X_2,\ldots,X_n\rangle$, i.e., the element of $\alg{M}$
  obtained after substituting every $X_i$ with $a_i$ for $i=1,2,\ldots,n$.

\end{notation}
\begin{definition}
  Fix a noncommutative probability space $(\alg{M},\psi)$.
  The \emph{noncommutative joint distribution} of a tuple of random variables
  $a_1,a_2,\ldots,a_n\in \alg{M}$ is the linear functional
  $\psi_{a_1,a_2,\dots,a_n}:\IC\langle X_1,\ldots,X_n\rangle\to\IC$
  given by the evaluation
  \[
    \psi_{a_1,a_2,\dots,a_n}(P(X_1,\ldots,X_n)):=\psi\left(P(a_1,a_2,\ldots,a_n)\right).
  \]
\end{definition}
\begin{remark}
  \label{rem:freealgebra}
  \begin{enumerate}[label=\arabic*., leftmargin=2em]
   \item []
   \item 
 The distribution of any polynomial $Q(a_1,\ldots,a_n)$ with respect to
  $\psi$ is the same as the distribution of $Q(X_1,\ldots,X_n)$ with respect to
  $\psi_{a_1,a_2,\dots,a_n}$,
  and it suffices to study the distribution of polynomials in
  $(\IC\langle X_1,\ldots,X_n\rangle,\psi_{a_1,a_2,\dots,a_n})$.
  This has the following advantages.
  \begin{enumerate}[label=(\roman*)]
   \item   The augmentation homomorphism  $\epsilon:\mathcal{M}\to\mathbb{C}$
    which maps a polynomial to its constant coefficient $\epsilon(P)=P(0)$
    allows the identification of the free product with the tensor algebra
    and thus the unambiguous translation of multilinear maps into
    linear maps on the free product.
   \item
    The derivations on the free algebra to be defined below
    turn the combinatorial recursions for Boolean cumulants and conditional expectations
    into concise algebraic identities.
  \end{enumerate}
  The disadvantage is a possible loss of faithfulness and positivity which has
  to be kept in mind.
  Nevertheless we will abuse notation and write $\psi$ instead of
  $\psi_{a_1,a_2,\dots,a_n}$, which should not lead to any confusion.

 \item
   For the sake of simplicity we will stick to the two-variable case,
  i.e., focus on the noncommutative probability space $\mathbb{C}\langle
  X,Y\rangle
  $ with states $\psi=\psi_{a,b}$ and $\varphi=\varphi_{a,b}$
The generalization to the multivariate case is straightforward and left to 
  the reader.

 \item
  It is straightforward to extend the formulas from polynomials
  to formal power series with polynomial coefficients and
  we will tacitly do this without further comment.
  \end{enumerate}
  
\end{remark}

\subsection{Boolean cumulant functionals and derivations}

It will be important to factor monomials into alternating blocks of ``pure''
monomials coming from different subalgebras.

\begin{definition}
  The \emph{support} of a polynomial is the set of letters which appear in its
  terms.
  The \emph{block factorization} of a monomial
  $W\in \IC \langle X,Y\rangle $ is the (unique) decomposition
  $W=U_1U_2\cdots  U_n$ into nontrivial monomials $U_i$, each supported either on $X$ alone or on $Y$ alone, alternating between the two.
The  block length of the monomial $W$ is then equal to the number $n$ of blocks. For example the block length of the word $W=X^3Y^2X$ is equal to 3.
\end{definition}

We can now define two kinds of linear functionals related to Boolean cumulants,
one which operates on blocks, and another which operates in letters. 
\begin{definition}
  \label{def:blockcumulant}
   The \emph{block Boolean cumulant functional} is the linear functional
  $\beta^{b,\psi}:\IC\langle  X,Y\rangle\to\mathbb{C}$
  determined by the following values on the monomial basis:
  \begin{equation}
    \label{eq:blockbeta}
    \begin{aligned}
      \beta^{b,\psi}(1)&=1\\
\beta^{b,\psi}(W)& = \beta_n^\psi(U_1,U_2,\ldots,U_n)
    \end{aligned}
  \end{equation}
  where  $W=U_1U_2\cdots U_n$ is the block factorisation of the monomial $W$.  
  Similarly we  define the  \emph{partial block Boolean cumulant functional}
  $\beta^{b,\psi}_X:\IC\langle X,Y\rangle\to\mathbb{C}$ as 
  \begin{align*}
    \beta^{b,\psi}_X(1)
    &= 1,
    \\ \beta^{b,\psi}_X(W)
    &=
      \begin{cases}
        \beta^{b,\psi}(W) & \text{if $W=X V X$ and
          $V\in \IC\langle X,Y\rangle$}\\
        $0$ & \text{otherwise}
      \end{cases}
  \end{align*}
  The linear functional  $\beta^{b,\psi}_Y:\alg{M}\to\mathbb{C}$ is defined analogously.
\end{definition}

For example we have $\beta^{b,\psi}_X(X^3Y^2X)=\beta_3^{\psi}(X^3,Y^2,X)$.

\begin{definition}
  The   \emph{letter-wise Boolean cumulant functional} is the linear functional
  $\beta^{\delta,\psi}:\alg{M}\to \IC$
  determined by the following values on the monomial basis:
  \begin{equation}
    \label{eq:betaDelta}
    \begin{aligned}
      \beta^{\delta,\psi}(1) &= 1 \\
    \beta^{\delta,\psi}(Z_1Z_2\cdots Z_k) &= \beta^\psi_k(Z_1,Z_2,\ldots,Z_k),
    \end{aligned}
  \end{equation}
  where $Z_i\in \{X,Y\}$ for $1\leq i\leq k$.
  In a similar way to $\bbetas{\psi}{X}$, we define the
  \emph{partial letter-wise Boolean cumulant functional}
  $\beta^{\delta,\psi}_X:\alg{M}\to\mathbb{C}$ on the monomial basis by
  \begin{align*}
    \beta^{\delta,\psi}_X(1)
    &= 1,
    \\
    \beta^{\delta,\psi}_X(Z_1Z_2\cdots Z_k)
    &= \left\{
      \begin{tabular}{l l }
        $\beta^{\psi}_k(Z_1,Z_2,\ldots,Z_k)$ & if $Z_1=Z_k = X$,
        \\ $0$ & otherwise
      \end{tabular}\right.
  \end{align*}
\end{definition}
For example we have $\beta^{\delta,\psi}_X(X^3Y^2X)=\beta_6^{\psi}(X,X,X,Y,Y,X)$.

\begin{remark}
  These cumulant functionals can be constructed for any state on $\IC\langle
  X,Y\rangle$, in particular we will
  consider these functionals $\beta^{b,\varphi}$ and $\beta^{\delta,\varphi}$
  for both states of a two-state noncommutative probability space.
  Note that if the joint distribution of $X$ and $Y$ satisfies Property (CAC)
  with respect to a state $\varphi$ then
  \begin{equation}
    \label{eq:beta=betaX+betaY}
    \bbetas{\varphi}{} =   \bbetas{\varphi}{X}+ \bbetas{\varphi}{Y}-\epsilon
    \qquad
    \fbetas{\varphi}{} =   \fbetas{\varphi}{X}+ \fbetas{\varphi}{Y}-\epsilon
    .
  \end{equation}
\end{remark}
The following derivations will be useful for  the algebraic
reformulation
of the recursive version     \eqref{eq:RecursiveBoolProd} of the  Leonov-Shiryaev formula
for Boolean cumulants with products as entries.
\begin{definition}
  \label{def:derivations}
  The
  \emph{free difference quotient}
  is the unique derivation such that
  $\partial_{X}(X)=1\otimes X$ and $\partial_{X}(Y)=0$.
  Thus we have
  \begin{equation*}
\partial_X (X^n)
    = \sum_{k=0}^{n-1} X^k\otimes X^{n-k-1}
    .
  \end{equation*}
  This operator is coassociative and we define
  \[
    \partial_X^2=(\id\otimes \partial_X)\circ \partial_X=( \partial_X \otimes
    \id)\circ \partial_X
  \]
  and similarly with higher powers $\partial_X^k$. 
  We will also need the \emph{free divided power derivations}
  or \emph{partial deconcatenation operators} defined by
  \begin{align*}
    \rdelta_{X}(P) &=  (1\otimes X)\partial_{X}(P),
    &
    \ldelta_{{X}}(P) &=  (X\otimes 1)\partial_{X}(P),
  \end{align*}
  i.e.,
  \begin{align*}
    \rdelta_X (X^n)
    &= \sum_{k=0}^{n-1} X^k\otimes X^{n-k}
    &
    \ldelta_X (X^n)
    &= \sum_{k=1}^{n} X^k\otimes X^{n-k}
  \end{align*}
\end{definition}

\begin{remark}
  Observe that the linear functionals $\bbeta{X}$, $\fbeta{X}$ and the
  derivations discussed above preserve their properties
  when they are extended to formal power series with
  polynomials as coefficients. 
\end{remark}

\subsection{Calculus for Boolean cumulant functionals}

The following theorem subsumes the functional relations between the maps
$\beta^{b,\psi}$ and $\beta^{\delta,\psi}$ in an algebraic way which will
allow us to evaluate these functionals effectively. 

\begin{proposition}
  \label{thm:FIC5.5}
  Assume that the random variables $X$ and $Y$ are free  with
  respect to the state $\psi$ on
  the  noncommutative probability space $\IC \langle X,Y\rangle$.
  \begin{enumerate}[label=(\roman*)]
   \item 
    For any element $P\in \IC\langle X,Y \rangle$ we have
    \begin{equation}
      \label{eq:intro_betas_prod_as_entries}
      \begin{aligned}
        \bbetas{\psi}{{X}}(P) 
        &= \epsilon(P) + (\fbetas{\psi}{{X}}\otimes\bbetas{\psi}{{X}})(\ldelta_{X} P)
        \\
        &= \epsilon(P) +
        (\bbetas{\psi}{{X}}\otimes\fbetas{\psi}{{X}})(\rdelta_{{X}} P)
        .
      \end{aligned}
    \end{equation}
   \item
    For any element $P\in \IC\langle X,Y \rangle$ we have  
\begin{equation}
      \label{eq:IntroVNRP_for_betas:uni}
      \fbetas{\psi}{X}(P)
      = \epsilon(P)+\sum_{k=1}^\infty
      \beta^\psi_k(X)
      \left[\epsilon\otimes
        \left(\bbetas{\psi}{ {Y}}\right)^{\otimes(k-1)}\otimes
        \epsilon\right]\left(
        \partial_X^k (P)
      \right).
    \end{equation}
   \item 
    If in addition  $X$ and $Y$ are $c$-free with respect to the pair
    of states $(\varphi,\psi)$ on
    $\mathbb{C}\langle X,Y\rangle$
    then \eqref{eq:intro_betas_prod_as_entries} holds for Boolean functionals
    with respect to $\varphi$ as well.
    Moreover for any element $P\in \IC\langle X,Y \rangle$ we have 
  \begin{equation}
    \label{eq:ProductsasEntriesCFree}
    \fbetas{\varphi}{X}(P)
    = \epsilon(P)+\sum_{k=1}^\infty
    \beta^\varphi_k(X)
    \left[\epsilon\otimes
      \left(\bbetas{\psi}{{Y}}\right)^{\otimes(k-1)}\otimes
      \epsilon\right]\left(
      \partial_X^k (P)
    \right).
  \end{equation}
  \end{enumerate}
\end{proposition}
\begin{proof}
  The statement concerning free variables follows from \cite[Theorem
  5.5]{LehnerSzpojankowski:2023:freeint1}. The proof for $c$-free variables
  follows the same arguments. Using the fact that the VNRP also holds
  for the $c$-free case as proved in Theorem~\ref{thm:VNTPCfree}. In
  particular, \eqref{eq:intro_betas_prod_as_entries} is a simple consequence
  of the Leonov-Shiryaev formula \eqref{eq:BoolProd}
  for $\beta^\varphi_n$, noticing that the extra
  terms vanish by Property (CAC) from Corollary~\ref{remdef:fCAC}. Equation
  \eqref{eq:ProductsasEntriesCFree} can be obtained by the same argument as in
  \cite[Theorem 5.5]{LehnerSzpojankowski:2023:freeint1}, except that, in 
  identity \eqref{eq:VNRP:cfree} of Theorem~\ref{thm:VNTPCfree}, the unique
  outer block is evaluated using $\beta^\varphi_n$ and the inner blocks are
  evaluated using $\beta^\psi_n$. 
\end{proof}
		
\subsection{Amplification to matrices}
Recall that the \emph{amplification} of a linear map $f:\alg{A}\to \alg{B}$
is the linear map
$f^{(N)}:= \id_{M_N}\otimes f:M_N(\alg{A})\to M_N(\alg{B})$,
i.e., the entry-wise application $f^{(N)}([a_{ij}]_{ij}) = [f(a_{ij})]_{ij}$.
This is then a matrix bimodule map in the sense that
\[
  f^{(N)}(UAV) = Uf^{(N)}(A)V
\]
for any $A\in M_N(\IC\langle X,Y\rangle)$ and scalar matrices
$U\in M_{k\times N}(\mathbb{C})$ and $V\in M_{N\times k}(\IC)$.
As discussed in \cite{LehnerSzpojankowski:2023:freeint1} the functionals
and derivations introduced above have amplifications to noncommutative
polynomials with matrix coefficients.
In particular,
$(M_N(\IC\langle X,Y\rangle),\psi^{(N)})$
is  a matrix-valued noncommutative probability space
and we will
identify $M_N(\IC\langle X,Y\rangle)
\cong M_N(\mathbb{C})\otimes \IC\langle X,Y\rangle$
and write  $AX$ resp.\ $BY$ for $A\otimes X$ resp.\
$B\otimes Y$.
 Observe that freeness of $X$ and $Y$ implies that for any $A,B\in
 M_{N\times N}(\mathbb{C}) $ the elements $AX$ and $BY$ are free
 with amalgamation over $M_{N}(\mathbb{C})$ in the matrix-valued probability
 space $(M_N(\IC\langle X,Y\rangle),\psi^{(N)})$.

When it is clear from the context that we are working in a matrix-valued
probability space, we will omit the superscript and write
$\varphi,\psi,\bbeta{X}$ and $\fbeta{X}$ for the matrix-valued functionals and
$I$ for the identity matrix $I_N\in M_N(\mathbb{C})$. 
		
\begin{remark}
  In the matrix-valued probability space $(M_N(\IC\langle X,Y\rangle),\psi^{(N)})$ for an element $A\in M_N(\IC\langle X,Y\rangle)$, we define its moment transform for $B\in M_N(\IC)$ by
  \[M_A^\psi(B)=I_N+\sum_{n=1}^\infty \psi^{(N)}\left((AB)^n\right),\]
  and similarly
  \[\eta_A^{\psi}(B)=\sum_{n=1}^\infty \beta_n^{\psi,(N)}\left((AB)^n\right).\]
  We still have
  \[M_A^\psi(B)=\left(I_N-\eta_A^{\psi}(B)\right)^{-1}.\]
  In what follows we will always assume that $B=zI_N$ for $z\in \mathbb{C}$.
\end{remark}
		
\begin{remnot}
  \begin{enumerate}[label=\arabic*., leftmargin=2em]
   \item []
   \item 
  The derivations $\partial_X,\rdelta_X, \ldelta_X$ introduced in the previous
  section can also be extended to the matrix-valued framework (see
  \cite[Cor.~6.6]{LehnerSzpojankowski:2023:freeint1}). We apply the respective
  derivations entry-wise to a matrix with polynomial coefficients.
The  Leibniz rule for the derivations from   \eqref{def:derivations}
  takes the   form 
  \[
    D(a\cdot b)=D(a) \cdot (1\otimes b)+(a\otimes 1) \cdot D(b),
  \]
  and  the amplified version is
  \[
    D^{(N)}(A\cdot B)=D^{(N)}(A) \cdot (1\odot B)+(A\odot 1) \cdot D^{(N)}(B),
  \]
  where for a matrix $A\in M_N(\mathbb{C})$ with entries $A=[a_{ij}]_{ij}$ 
  we denote $1\odot A=[1\otimes a_{ij}]_{ij}$
  and 
  $A\odot 1=[ a_{ij} \otimes 1]_{ij}$.
 \item    A simple argument using Leibniz' rule shows that for invertible matrices
  \begin{equation}
    \label{eq:DA-1}
    D^{(N)}(A^{-1}) = -(A^{-1}\odot 1) D^{(N)}(A) (1\odot A^{-1})
    .
  \end{equation}
 \item
  More generally, we define $A\odot B= (A\odot 1)\cdot(1\odot B)$, i.e.,
  \[
    (A\odot B)_{ij}=\sum_{k=1}^N a_{ik} \otimes b_{kj}.
  \]
 \item 
  The entry-wise application of linear functionals then turns $\odot$ into a simple
  multiplication of scalar   matrices:
  $$
  (f\otimes g)^{(N)} (A\odot B) = f^{(N)}(A)\cdot g^{(N)}(B)
  $$
  and the formulas
  \eqref{eq:intro_betas_prod_as_entries}--\eqref{eq:ProductsasEntriesCFree}
  from  Proposition~\ref{thm:FIC5.5} remain valid after amplification to the
  matrix valued case.
\end{enumerate}

\end{remnot}

\subsection{Matrix-valued $c$-free additive convolution}

Assume that we have a polynomial $P\in \mathbb{C}\langle X,Y\rangle$ of degree 
$m$ together with a linearization of its resolvent:
$$
(1-z^mP)^{-1} = u^t(I_N-zL)^{-1}v,
$$
where $u,v\in \mathbb{C}^N$ and $L=A X+BY$. Here $(1-z^mP)^{-1}=\sum_{n=0}^{\infty} (z^mP)^n$ is a formal power series with coefficients in $\mathbb{C}\langle X,Y\rangle$ and similarly $(I_N-zL)^{-1}$ is a formal power series with coefficients in $M_N(\mathbb{C}\langle X,Y\rangle)$. 
Then
\[\varphi\left((1-z^m
    P)^{-1}\right)=u^t\varphi^{(N)}\left((I_N-zL)^{-1}\right)v
\]
and since 
the moment generating function uniquely determines the distribution of $P$, the
problem of computing the distribution of a polynomial $P$ boils down to
matrix-valued additive convolution.
This is done via the matrix valued Boolean cumulant generating functions which
satisfy the fixed point equation featuring in the following Proposition.

\begin{proposition}
  \label{lem:SeveralEquations}
  Assume that $X$ and $Y$ are $c$-free in the two-state noncommutative
  probability space $(\mathbb{C}\langle X,Y\rangle,\varphi,\psi)$.
  Given matrices  $A,B\in M_N(\mathbb{C})$,
  denote by $\bPsi = \bigl(I - z(AX+BY)\bigr)^{-1}$ the resolvent and let
  \begin{align}
    H_X & :=\bbetas{\psi}{X}(\bPsi),
    &
      F_X & :=\fbetas{\psi}{X}(X\bPsi), \\
    H_Y & :=\bbetas{\psi}{Y}(\bPsi),
    &
      F_Y & :=\fbetas{\psi}{Y}(Y\bPsi) ,
  \end{align}
  \begin{align}
    F_X^\varphi & :=\fbetas{\varphi}{X}(X\bPsi), \\
    F_Y^\varphi & :=\fbetas{\varphi}{Y}(Y\bPsi).
  \end{align}
  Then the following identities hold:
  \begin{align}
    H_X &=  (I-zAF_X)^{-1},\label{eq:HAdditive}\\
    H_Y &= (I-zBF_Y)^{-1},\label{eq:HYAdditive}\\
    F_X &= \tilde{\eta}_X^{\psi}\left(z H_Y A \right),\label{eq:Fx=betasubordination}\\
    F_Y &= \tilde{\eta}_Y^{\psi}\left(zH_X B \right).
  \end{align}
  \begin{align}
    F_X^\varphi &= \tilde{\eta}^\varphi_X(zH_YA),\label{eq:Fxphi=betasubordination}\\
    F_Y^\varphi &= \tilde{\eta}^\varphi_Y(zH_XB)\label{eq:Fyphi=betasubordination}.
  \end{align}
\end{proposition}
\begin{remark}
  In Proposition~\ref{lem:SeveralEquations}
  we have assumed for the sake of simplicity that
  the coefficient matrices $A$ and $B$ are scalar matrices.
  In practice the linearization of nonhomogeneous polynomials
  leads to matrices $A$ and $B$ with nonlinear dependence on $z$,
  however it is easy to see that all formulas remain valid.
\end{remark}
\begin{remark}
  The matrices
  \begin{align*}
    H_X^\varphi & :=\bbetas{\varphi}{X}(\bPsi), 
    \\
    H_Y^\varphi & :=\bbetas{\varphi}{Y}(\bPsi), 
  \end{align*}
  satisfy the identities
  \begin{align}
    \label{eq:HXphi=1-zBFXphi}
    H_X^\varphi &=  (I-zAF_X^\varphi)^{-1},\\
    \label{eq:HYphi=I-zAFXphi}
    H_Y^\varphi &=  (I-zBF_Y^\varphi)^{-1},
  \end{align}
  similar to the identities   \eqref{eq:HAdditive}
  and   \eqref{eq:HYAdditive} satisfied by $H_X$ and $H_Y$,
  however they are not needed because we can obtain
  the matrices $F^\varphi_X$ and
  $F^\varphi_Y$ directly from  on $H_X$ and $H_Y$ via formulas
  \eqref{eq:Fxphi=betasubordination} and \eqref{eq:Fyphi=betasubordination}.
  In fact this was already observed in \cite {Belinschi:cfree} to the effect
  that the very same  subordination  functions underlying free additive
  convolution keep their role in the subordination approach to $c$-free
  additive convolution.
\end{remark}
\begin{proof}
For the reader's convenience, we present a proof of \eqref{eq:HAdditive} and
  \eqref{eq:Fxphi=betasubordination}.
  The proofs of the remaining statements are entirely similar.
  By definition of $\bPsi$ and $\rdelta_X$, we have
  $$
  \rdelta_X(\bPsi) = z\bPsi A \odot X\bPsi = z\bPsi \odot AX\bPsi.
  $$
  Then the matricial amplification of the Leonov-Shiryaev formula
  \eqref{eq:intro_betas_prod_as_entries}
  reads
  $$
  \beta^{b,\psi}_X(\bPsi) = I + z\beta_X^{b,\psi}(\bPsi)\beta_X^{\delta,\psi}(AX\bPsi),
  $$
  and hence
  $$
  H_X=\beta_X^{b,\psi}(\bPsi) = \left(I-zA\beta_X^{\delta,\psi}(X\bPsi)\right)^{-1} = (I-zAF_X)^{-1}.
  $$
  Finally, it is not difficult to show that
  \[
    \partial_X^n(X\bPsi) = z^{n-1} I\odot (\bPsi A)^{\odot n-1}\odot \bPsi + z^n X(\bPsi A)\odot (\bPsi A)^{\odot n-2}\odot \bPsi
  \]
  and after plugging this into the recurrence \eqref{eq:ProductsasEntriesCFree} we conclude
  \begin{align*}
    F_X^\varphi = \beta_X^{\delta,\varphi}(X\bPsi) &= \sum_{k=1}^\infty \beta_k^\varphi(X)z^{k-1} \beta^{b,\psi}_Y(\bPsi A)^{k-1}
    \\ &= \tilde{\eta}_X^{\varphi}\bigl( z \beta_Y^{b,\psi}(\bPsi)A\bigr).
  \end{align*}
\end{proof}

The solutions of the previous fixed point equations can now be used to
compute  distributions of polynomials in $c$-free variables via linearization
as shown in the next theorem.
\begin{theorem}
  \label{thm:cfreeConvolution}
  With the notation from Proposition~\ref{lem:SeveralEquations},
  assume that $X$ and $Y$ are $c$-free
  in the   two-state noncommutative probability space  $(\mathbb{C}\langle X,Y\rangle,\varphi,\psi)$.
  Assume moreover that the resolvent of a given polynomial $P \in \mathbb{C}\langle X,Y\rangle$ of degree $m$ has a linearization given by $\Psi = (1-z^mP)^{-1} = u^t\bPsi v$, where $\bPsi = \bigl(I-z(AX+BY)\bigr)^{-1}$, with $A,B\in M_N(\mathbb{C})$ and $u,v\in \mathbb{C}^N$. Then
  \begin{equation}
    M^\varphi_{AX+BY}(z):=\varphi(\bPsi)
    = (I-zAF^\varphi_X-zBF_Y^\varphi)^{-1}
    \label{eq:AX+BY:phi(Psi)}
  \end{equation} and
  consequently
  \[M^{\varphi}_{P(X,Y)}(z^m)=u^t M^\varphi_{AX+BY}(z) v.\]
\end{theorem}

\begin{proof}
By the definition of the resolvent, observe that
  \begin{equation}
    \label{eq:resolventIdentity}
    \bPsi = I + zAX\bPsi + zBY\bPsi.
  \end{equation}
  Then, we have\begin{align*}
    \eta^\varphi_{AX+BY}(z)
    &=\sum_{n=1}^{\infty}\beta_n^{\varphi}(AX+BY) z^n\\
    &=\sum_{n=1}^{\infty}\beta^{\delta,\varphi}\left((AX+BY)^n\right) z^n\\
    &= \beta^{\delta,\varphi}(\bPsi) - I\\
    &= \beta^{\delta,\varphi}(zAX\bPsi+zBY\bPsi)\\
    &= zA\beta_{X}^{\delta,\varphi}(X\bPsi)+zB\beta_{Y}^{\delta,\varphi}(Y\bPsi)\\
    &= zAF^\varphi_X+zBF^\varphi_Y,
  \end{align*}
  where we simply use the definition of $\eta^\varphi$,
  properties of $\beta^\delta$ and
\eqref{eq:beta=betaX+betaY}.
  The result follows from the identity $M^\varphi(z)=(I-\eta^\varphi(z))^{-1}$.
\end{proof}
		
\section{Applications to free denoising}
\label{sec:denoising}
We will now apply the above result to the problem of free denoising. First, we briefly recall the framework of free denoising. 

\subsection{Review of free denoising}
Consider a tracial $W^*$-probability space $(\mathcal{M},\psi)$. Fix two
self-adjoint elements $a,b\in \mathcal{M}$. We view $a$ as a signal and $b$ as
a noise. For a self-adjoint polynomial $P\in\mathbb{C}\,\langle X,Y\rangle$ we
consider the element $P(a,b)$, which represents a noisy element. Our goal is to
recover the signal $a$ from the noisy element $P(a,b)$, and we observe that
$\E[a|P(a,b)]$ solves this problem in the sense of $L^2$ distance, i.e.,
$\lVert a-\E[a|P(a,b)]\rVert_2\leq \lVert a-g(P(a,b))\rVert_2$ among all bounded
Borel functions $g$. We will work with a slightly more general problem and we
will consider $\E[f(a)|P(a,b)]$. The crucial observation from \cite{FevrierNicaSzpojankowski:2024} is that $c$-freeness provides a framework to find $\E[f(a)|P(a,b)]$. Fix a non-negative function $f$ such that $\psi(f(a))=1$ and define a new state $\varphi(c)=\psi(f(a)c)$. The main observations from \cite{FevrierNicaSzpojankowski:2024} are the following:
\begin{enumerate}[label=\arabic*.]
 \item The elements $a,b$ are $c$-free with respect to the pair $(\varphi,\psi)$.
 \item Given the distributions of $P$ with respect to both states
  $\mu_{P(a,b)}^{\psi}$ and $\mu_{P(a,b)}^{\varphi}$, the measure
  $\mu_{P(a,b)}^{\varphi}$ is absolutely continuous with respect to
  $\mu_{P(a,b)}^{\psi}$, and therefore there exists the Radon-Nikodym
  derivative $\tfrac{d \mu_{P(a,b)}^{\varphi}}{d \mu_{P(a,b)}^{\psi}}$ which we
  denote by $h$. 
 \item The Radon-Nikodym derivative above coincides with the conditional
  expectation
  in the sense that
  \[
    \E[f(a)|P(a,b)]=h(P(a,b)).
  \]
\end{enumerate}
This reduces the problem of free denoising to the computation of the
distributions of
polynomials in
free variables and $c$-free variables. In \cite{FevrierNicaSzpojankowski:2024}
this general framework is provided, however it was not possible to apply it
beyond addition and multiplication of variables for lack of a general
method to determine distributions of polynomials in $c$-free
variables. Theorem~\ref{thm:cfreeConvolution} fills this gap.

\subsection{A worked out example}
Given a noncommutative probability space $(\alg{M},\psi)$,
let $a$ be a standard semicircular element and
$b$ free from $a$ with distribution $\frac{1}{2}(\delta_{-1}+\delta_1)$.
In this subsection we present a worked out example which shows how to find for fixed $-1/2<s<1/2$ the conditional expectation
\[
  \E\bigl[(1-sa)^{-1} \big| i[a,b]\bigr],
\]
where $i[a,b] = i(ab - ba)$ is the commutator of $a$ and $b$. Define a state
\[\varphi(c):=\frac{\psi\left((1-sa)^{-1}c\right)}{M_a^{\psi}(s)}.\]
As we observed above, we need to determine the distribution of $i[a,b]$ with
respect to both states $\psi$ and $\varphi$.
As discussed in Subsection~\ref{ssec:nc_polynomials},
we will work on the free algebra $\left(\C\langle
  X,Y\rangle,\varphi_{a,b},\psi_{a,b} \right)$
where we have $\varphi_{a,b}(R(X,Y))=\varphi(R(a,b))$ and
$\psi_{a,b}(R(X,Y))=\psi(R(a,b))$
and after fixing $a$ and $b$ we will simply write $\varphi, \psi$
instead of $\varphi_{a,b},\psi_{a,b}$, which should not lead to any confusion.

For $z$ in some neighbourhood of zero we have
\begin{align*}
  M_X^{\varphi}(z)&=\varphi\left((1-zX)^{-1}\right)=\frac{\psi\left((1-sX)^{-1}(1-zX)^{-1}\right)}{M_X^{\psi}(s)}\\
                  &=\frac{\psi\left(\frac{s}{1-s X}-\frac{z}{1-z X}\right)}{(s-z)M_X^{\psi}(s)}=\frac{sM_X^\psi(s)- zM_X^\psi(z)}{(s-z)M_X^\psi(s)},
\end{align*}
or,  equivalently,
\begin{align}\label{eqn:61}
  \tilde{\eta}^{\varphi}_X(z)=\frac{M^{\psi}_X(s)-M_X^{\psi}(z)}{sM^{\psi}_X(s)-zM_X^{\psi}(z)}.
\end{align}
The linearization is
\[\Psi(z^2)=(1-z^2 i[X,Y])^{-1}=\left[(1-z C_X X-z C_Y Y)^{-1}\right]_{1,1} \]
with
\begin{align*}
  C_X=
  \begin{bmatrix}
    0&0&i\\
    1&0&0\\
    0&0&0
  \end{bmatrix}
         \qquad
         C_Y=
         \begin{bmatrix}
           0&-i&0\\
           0&0&0\\
           1&0&0
         \end{bmatrix}
\end{align*} 
and  we can construct the system of equations for matrices $F_X,F_Y,F_X^{\varphi}, F_Y^{\varphi}$ as follows:
\begin{align*}
  \begin{cases}
    H_X&=(I-z C_X F_X)^{-1},\\
    H_Y&=(I-z C_Y F_Y)^{-1},\\
    F_X&=\tilde{\eta}^{\psi}_X(z H_Y C_X),\\
    F_Y&=\tilde{\eta}^{\psi}_Y(z H_X C_Y),\\
    F_X^{\varphi}&=\tilde{\eta}_X^{\varphi}(z H_Y C_X),\\
    F_Y^{\varphi}&=\tilde{\eta}_Y^{\varphi}(z H_X C_Y).\\
  \end{cases}
\end{align*}
Note that by the freeness assumption, the distributions of  $Y$ with respect to
$\varphi$ and $\psi$ coincide and hence
$\tilde{\eta}^{\varphi}_Y(z H_X C_Y)=\tilde{\eta}^{\psi}_Y(z H_X C_Y)=zH_X
C_Y$,
where the latter equality follows from the fact that $b$ has the Bernoulli
distribution.
As a consequence $F_Y^{\varphi}=F_Y$ as well.

Since $a$ was assumed to be the standard semicircular element, its moment generating function satisfies the equation \[M_X^{\psi}(z)=z^2 M_X^{\psi}(z)^2+1.\]
For the $\eta$-transform we get $\eta_X^\psi(z)=1-1/M_X^\psi(z)$. This implies that $\tilde{\eta}_X^{\psi}=z M_X^{\psi}(z)$. This together with \eqref{eqn:61} implies that
\[
  \tilde{\eta}^{\varphi}_X(z)=\frac{s^2 M_X^{\psi}(s)^2-z^2 M_X^{\psi}(z)^2}{s
    M_X^{\psi}(s)-z M_X^{\psi}(z)}=s M_X^{\psi}(s)+z
  M_X^{\psi}(z)=\tilde{\eta}_X^{\psi}(s)+\tilde{\eta}_X^{\psi}(z)
\]
and finally
\[F_X^{\varphi}=\tilde{\eta}_X^{\psi}(s)I+\tilde{\eta}_X^{\psi}(z H_Y C_X)=\tilde{\eta}_X^{\psi}(s)I+F_X.\]
 In particular we have
\[\eta_{C_X X+C_Y Y}^{\varphi}(z)=z C_X F_X^{\varphi}+z C_Y F_Y^{\varphi}=z C_X \tilde{\eta}^\psi_X(s)+\eta_{C_X X+C_Y Y}^{\psi}(z)\]
and solving this system (and noticing that the involved moment and $\eta$-transforms have analytic continuation to the upper half-plane), for $z\in \IC^+$ we get
\[G^{\varphi}_{i[a,b]}(z)=G^{\varphi}_{i[X,Y]}(z)=\frac{(2+\tilde{\eta}_X^2(s))\sqrt{z^2-8}+(\tilde{\eta}_X^2(s)-2)z}{2\tilde{\eta}_X^2(s) z^2-2(\tilde{\eta}_X^2(s)+2)^2 }.\]
Stieltjes inversion formula and a direct calculation gives 
\[d\mu^{\varphi}_{i[a,b]}(x)=-\frac{1}{\pi}\frac{(2+\tilde{\eta}_X^2(s))\sqrt{8-x^2}}{2\tilde{\eta}_X^2(s) x^2-2(\tilde{\eta}_X^2(s)+2)^2}.\]
Since it is well known that the distribution of the commutator of a semicircular
variable and a Bernoulli variable is equal to the free convolution of two semicircle distributions of variance one, we obtain the density $\tfrac{1}{4 \pi}\sqrt{8-x^2}$.
We finally obtain the Radon-Nikodym derivative and the conditional expectation
\[h(t)=M_a^\psi(s)\frac{d\mu^{\varphi}_{i[a,b]}}{d\mu^{\psi}_{i[a,b]}}=\frac{2(2+\tilde{\eta}_X^2(s))}{(1-s\tilde{\eta}_X(s))\left((2+\tilde{\eta}_X(s))^2-\tilde{\eta}_X^2(s) t^2\right)}.\]
Extracting the first few coefficients from the generating function we get
\begin{align*}
  \E[a^2|i[a,b]]&=\frac{(i[a,b])^2+2}{4},\\
  \E[a^4|i[a,b]]&=\frac{(i[a,b])^4+6(i[a,b])^2+12}{16}.
\end{align*}
		
\section{Algebraic conditional expectations of $c$-free variables}\label{sec:6}
\subsection{Conditional expectations}
In this section, we study formal analogues of conditional expectations for
$c$-free random variables.
It turns out that after some modifications the algebraic machinery from
\cite{LehnerSzpojankowski:2023:freeint1} still works in the $c$-free setting,
and we obtain projections onto subalgebras which formally
act like conditional expectations,
but lack important analytic properties like positivity.
Yet they are useful for practical computations and as an example
in Section~\ref{sec:sigmatransform} we recompute the
$c$-free analog of the $\Sigma$-transform  from \cite{PopaWang:2011:multiplicative}
by purely algebraic means.

Recall that a \emph{conditional expectation}
is a state-preserving projection $\E_{\alg{A}}:\alg{M}\to \alg{A}$
from a noncommutative probability space $\alg{M}$ onto a subalgebra $\alg{A}$
such that in addition
for every $u\in\alg{M}$ and
every $a_1,a_2\in\alg{A}$
\begin{align}
  \label{eq:phiua=phiEua}
  \varphi(a_1u)&=\varphi(a_1\E_{\alg{A}}[u]),
  &
    \varphi(ua_2)
  &=\varphi(\E_{\alg{A}}[u]a_2)
    .
\end{align}
In general such a map does not necessarily exist.
A sufficient condition is that $\alg{M}$ is a finite
von Neumann algebra and $\varphi$ is tracial,
i.e.~$\varphi(ab) =\varphi(ba)$ for all $a,b\in \alg{M}$
\cite[Proposition~5.2.36]{Takesaki1};
for a characterization in the general case see \cite {Takesaki:1972}.
If it does exist and the state $\varphi$ is faithful, then the conditional expectation
$\E_{\alg{A}}[u]$ is uniquely determined by the invariance property \eqref{eq:phiua=phiEua}
and it follows that
$\E_{\alg{A}}:\alg{M}\to\alg{A}$ is a unital completely positive $\alg{A}$-bimodule map, i.e.,
\begin{equation}
  \label{eq:modularproperty}
  \E_{\alg{A}}[a_1ua_2]=a_1\E_{\alg{A}}[u]a_2 
\end{equation}
for all $u\in\alg{M}$ and $a_1,a_2\in\alg{A}$.

A note on notation:
The standard notation for the conditional expectation onto a subalgebra
$\alg{A}$ is $\E_{\alg{A}}$ and when the subalgebra  $\alg{A}$ is generated
by a single element $a$ we also write $\E_{a}$ 
Occasionally we write the latter also as $E[x|a]=E_a[x]$,
for example in Section~\ref{sec:denoising} in order to avoid clumsy expressions like
$\E_{i[a,b]}$.

\subsection{Left and right quasi-conditional expectations}
\label{ssec:ExistenceCondExp}
From a technical point of view,
the computation of the conditional expectation
of an element $u\in\alg{M}$
onto a subalgebra $\alg{A}\subseteq\alg{M}$
boils down to the construction
of an element $E_{\alg{A}}[u]\in \alg{A}$ which satisfies
the conditions \eqref{eq:phiua=phiEua}.
This amounts to the computation of the orthogonal projection
$L^2(\alg{M},\varphi)\to L^2(\alg{A},\varphi)$.

Let $(\alg{M},\varphi)$ be a $W^*$-probability space with faithful state $\varphi$, $\alg{A}\subseteq\alg{M}$
a subalgebra and $L^2(\alg{M})_r$ and $L^2(\alg{A})_r$  the right $L^2$-spaces,
i.e., the respective completions
of $\alg{M}$ and $\alg{A}$ with respect to the scalar product
\begin{equation*}
  \langle x,y\rangle  = \varphi(xy^*)
  .
\end{equation*}
Then the orthogonal projection $\rE_{\alg{A}}:L^2(\alg{M},\varphi)_r\to L^2(\alg{A},\varphi)_r$ exists
and satisfies the right module property $\rE_{\alg{A}}[ua]=\rE_{\alg{A}}[u]a$
for $u\in\alg{M}$ and  $b\in\alg{A}$.
Similarly,
the left Hilbert spaces $L^2(\alg{M},\varphi)_l$ and $L^2(\alg{A},\varphi)_l$
are obtained from the scalar product
\begin{equation*}
  \langle x,y\rangle  = \varphi(y^*x)
\end{equation*}
and the projection  $\lE_{\alg{M}}:L^2(\alg{M},\varphi)_l\to L^2(\alg{A},\varphi)_l$ satisfies
the left module property $\lE_{\alg{A}}[au]=a\lE_{\alg{B}}[u]$.
Note that it is not necessarily true
that  the restriction of a map  $F:L^2(\alg{M})_r\to L^2(\alg{A})_r$
to the subspace $\alg{A}\subseteq L^2(\alg{M})_r$ automatically
gives rise to a map $F:\alg{M}\to\alg{A}$.
In the $c$-free setting below both $\alg{M}$ and $\alg{A}$ will be algebras of
polynomials and we will show that the projections
indeed restrict to mappings from $\alg{M}$ to $\alg{A}$.
We will call these mappings \emph{left} and \emph{right quasi-conditional expectations}.

Another issue to note here is faithfulness: if the state $\varphi$ is not faithful,
then the scalar product is not positive definite and the projection not unique,
i.e., it is only defined up to a one-sided ideal.
However we will simply ignore this issue
and define a formally consistent
map
satisfying all essential algebraic requirements needed for the actual computations.

\subsection{Failure of positivity and the modular property}
\label{ssec:failurepositive}
The following examples show that in the case of $c$-freeness the quasi-conditional expectations with
respect to $\varphi$ are not positive and fail the modular property
\eqref{eq:modularproperty}.
More precisely,
positive elements are not necessarily mapped to positive elements
the left modular property does
not hold for the right  quasi-conditional expectation and vice versa.

\begin{example}[Failure of the left modular property]
Assume $\mathcal{A}$ and $\mathcal{B}$ are $c$-free subalgebras in
  $(\alg{M},\varphi,\psi)$.
  Then for $a\in\mathcal{A}$ and $b\in \mathcal{B}$  we have
  \begin{equation*}
    \varphi(ba^n) = \varphi(b)\varphi(a^n)
  \end{equation*}
  for every $n\in\IN$ 
  and consequently the right conditional expectation onto $\alg{A}$ is
  \begin{equation*}
    \rE^\varphi_\mathcal{A}[b] = \varphi(b)1_\alg{M}.
  \end{equation*}
  On the other hand, the conditional free pyramidal law asserts
  \begin{equation*}
    \varphi(a_1ba_2)
    = \varphi(a_1a_2)\psi(b) + \varphi(a_1)\varphi(a_2)(\varphi(b)-\psi(b))
  \end{equation*}
  i.e.,
  \begin{align*}
    \varphi(aba^n)
    &= \varphi(a^{n+1})\psi(b) + \varphi(a)\varphi(a^n)(\varphi(b)-\psi(b))\\
    &= \varphi( (a\psi(b) + \varphi(a)(\varphi(b)-\psi(b)))a^n)
  \end{align*}
  and consequently
  \begin{equation}
    \label{eq:rEab}    
  \begin{aligned}
    \rE^\varphi_\alg{A}[ab]
    &= a\psi(b) + \varphi(a)(\varphi(b)-\psi(b))1_\alg{M}\\
    &\ne a\rE^\varphi_\mathcal{A}[b] = a\varphi(b)
  \end{aligned}
  \end{equation}
  unless $\varphi(b)=\psi(b)$.
\end{example}

\begin{example}[Failure of positivity]
The example at the origin of conditional freeness has its roots
  in harmonic analysis of the free group
  \cite{BozejkoLeinertSpeicher:1996:convolution}.
  Let $\Gamma=\IZ_2*\IZ_2 = \langle a,b\mid a^2=b^2=1\rangle$
  be the free Coxeter group with two generators.
  Then $a$ and $b$ are free with respect to the von Neumann trace
  $\psi(w) = \delta_{e,w}$ and $c$-free with respect
  to the regular state $\varphi(w)=e^{-\lambda\abs{w}}$ with $0<\lambda<\infty$,
  where by $\abs{w}$ we denote length of a (reduced) word $w\in\Gamma$.
  These functionals define states on the group von Neumann algebra
  $L(\Gamma)$.
  Let $x=1+a$ and $y=b$. Then $x$ is positive and so is $bxb$,
  however plugging $bxb$ into     \eqref{eq:rEab} we obtain
  $$
  \rE^\varphi_\alg{B}[bxb] = b^2+(1+e^{-\lambda})e^{-\lambda}b
  $$
  which is not positive for sufficiently small $\lambda$.
\end{example}

Note however, that in the context of denoising in Section~\ref{sec:denoising}
the left and right conditional expectations coincide with the free one.
\begin{proposition}
  Let $f$ be a strictly positive function on the spectrum of $a$
  such that $\psi(f(a))=1$.
  Then the state $\varphi(x)=\psi(f(a)x)$ is faithful
  and  the right quasi-conditional expectation $\rE_a^{(\varphi)}$ coincides with
  the free conditional expectation $E^{\psi}_a$.
\end{proposition}
\begin{proof}
  Indeed,
  \begin{align*}
    \varphi\bigl(E_a^\psi[x]g(a)\bigr)
    &=     \psi\bigl(f(a)E_a^\psi[x]g(a)\bigr)\\
    &=     \psi(f(a)xg(a))\\
    &=     \varphi(xg(a))
  \end{align*}
  and therefore $\rE^\varphi_a[a] = E_a^\psi[a]$.  
\end{proof}
\begin{remark}
This does not contradict \eqref{eq:rEab} as we have $\varphi(b)=\psi(b)$ when
$b$ is free from $a$ with respect to $\psi$.  
\end{remark}

\subsection{A recurrence for the right quasi-conditional expectation}
Before embarking on the $c$-free problem,
we recall the algebraic framework for the free conditional expectation from
\cite{LehnerSzpojankowski:2023:freeint1}.

Recall that for the construction of the conditional expectation of
a noncommutative polynomial $u$
in variables $a_1,a_2,\ldots,a_n\in\alg{M}$ 
onto the algebra $\alg{A}$ generated by a subset of the variables
$a_1,a_2,\ldots,a_k$  one has to find suitable expressions for moments of the form 
\begin{equation*}
\varphi\left(a_{i_1}a_{i_2}\dotsm a_{i_r} b\right),
\end{equation*}
where $i_1,i_2,\ldots,i_r\in \{1,\ldots,n\}$ and $b\in\alg{A}$.
It is a fundamental property of freeness
(as one of the  universal notions of independence in the sense of \cite{Muraki:2002:five})
that all joint moments of
freely independent random variables are uniquely determined by the marginal moments
of the variables in question. Thus for each moment of the form indicated above
there is a universal formula (not depending on the particular choice of the
distributions of $a_1,a_2,\ldots,a_n$) which expresses any joint moment as a sum
of products of marginal moments.

  We will work in the formal setting of noncommutative polynomials
  $\IC\langle X,Y\rangle$   with $c$-free state $\varphi$ and free state $\psi$
  as described in Subsection~\ref{ssec:nc_polynomials}.
  The states are possibly not faithful and the conditional expectation
  therefore not unique, yet the following formula
consistently defines a   conditional expectation
  with respect to $\psi$ and specifies to the correct one when
  evaluated on free elements in von Neumann algebras.
  \begin{proposition}[\cite{LehnerSzpojan}]
    \label{prop:freecondexpX}
    Assume that $\psi$ is a state on $\IC\langle X,Y\rangle$ such that $X$ and
    $Y$ are free.
    Given a word  in block-factorization $W=X_0Y_1X_1Y_2\dotsm X_{n-1}Y_nX_n$
    (with $X_0$ and $X_n$ possible being empty),
    we define the linear map $\E^\psi_X:\IC \langle X,Y \rangle\to \IC \langle X \rangle$
    by
      \begin{multline}
      \label{eq:Epsi:full}
\E^\psi_{X}\left[X_0Y_1X_1Y_2\dotsm X_{n-1}Y_nX_n\right]
        \\
  =
    \sum_{p=0}^{n-1}\sum_{0=i_0 < i_1<i_2<\dots< i_{p+1}= n}
    X_{0}X_{i_1}\dotsm X_{i_{p}}X_n\prod_{j=0}^p \beta^\psi_{2(i_{j+1}-i_j)-1}(Y_{i_j+1},X_{i_j+1},\ldots, Y_{i_{j+1}}).        
      \end{multline}      
Then $E_X^\psi$ is  a conditional expectation with respect to the state $\psi$,
    i.e., it has the bimodule property $E_X^\psi[X_0WX_n]=X_0E_X^\psi[W]X_n$
    and preserves the state $\psi$.
\end{proposition}

The expansion \eqref{eq:Epsi:full} has a diagrammatic representation similar to VNRP as in
Remark~\ref{rmk:VNRPDiagrammatic}:
\begin{multline}
  \label{eq:Epsi:full:pic}
  E^\psi_X[X_0 Y_1 X_1 \dotsm Y_n X_n]
  \\=
  \sum
\begin{tikzpicture}[anchor=base,baseline]
\draw (-1.5,-0.4)--(-1.5,0.75)--(10.5,0.75)--(10.5,-0.4);
    \draw (1.5,-0.4)--(1.5,0.75);
    \draw (4.5,-0.4)--(4.5,0.75);
    \draw(7.5,-0.4)--(7.5,0.75); 		
    \node at (-1.5,-0.8){$X_0$};
    \node at (1.5,-0.8){$X_{i_1}$};
    \node at (4.5,-0.8){$X_{i_2}$};
    \node at (7.5,-0.8){$X_{i_p}$};
    \node at (10.5,-0.8){$X_{n}$};
    \paczek{0.0,0.0}{\beta^\psi(W_0)}
    \paczek{3,0.0}{\beta^\psi(W_1)}
    \node at (6,0.0){$\cdots$};
    \paczek{9,0.0}{\beta^\psi(W_p)}
  \end{tikzpicture}
\end{multline}
where $W_j=Y_{i_j+1}X_{i_j+1}\dotsm Y_{i_{j+1}}$.

The main tool in \cite{LehnerSzpojankowski:2023:freeint1} is the following
recurrence for the conditional expectation $E_X^\psi$, when evaluated on
monomials starting and ending in $Y$, i.e., when $X_0=X_n=1$ in formula 
\eqref{eq:Epsi:full} above:
\begin{equation}
  \label{eq:Epsi:recurrence}
  E^\psi_{X}[Y_0X_1Y_1\dotsm X_nY_n] =
  \sum_{k=0}^n\beta^\psi_{2k+1}(Y_0,X_1,Y_1,\dots, X_k,Y_k) \,  E^\psi_{X}[X_{k+1}Y_{k+1}\dotsm X_nY_n]
  ,
\end{equation}

which follows a pattern similar to \eqref{fig:boolRecPsi}
$$
E^\psi_X[Y_0 X_1 Y_1 \cdots X_n Y_n]
=
\sum
\begin{tikzpicture}[anchor=base,baseline]
  \node at (-0.25,0.45) {$\beta^{\psi}$};
  \draw (-0.25,-0.25)--(-0.25,0.4)--(3.25,0.4)--(3.25,-0.25);
  \draw (0.25,-0.25)--(0.25,0.4);
  \draw (0.75,-0.25)--(0.75,0.4);
  \draw (1.25,-0.25)--(1.25,0.4);
  \draw (2.75,-0.25)--(2.75,0.4);
  \node at (2.0,-0.1){$\cdots$};
  \node at (-0.25,-0.7){$Y_0$};
  \node at (0.25,-0.7){$X_1$};
  \node at (0.75,-0.7){$Y_1$};
\node at (3.3,-0.7){$Y_{k}$};
\end{tikzpicture}
\kern-0.5em
E^\psi_X[X_{k+1}Y_{k+1}  \dots X_n Y_n]
$$
Furthermore in \cite{LehnerSzpojankowski:2023:freeint1} the second and third named
authors have described recursive relations between the functionals $\E_X^\psi$
and $\beta_Y^{b,\psi}$ in terms of partial deconcatenations.

\begin{theorem}[{\cite[Thm.~4.12]{LehnerSzpojankowski:2023:freeint1}}]
  \label{thm:FreeIntCalc4.12}
  For any monomial $W\in\IC\langle X,Y\rangle$  the following recurrence relations hold:
  \begin{equation} 
    \label{eq:condexprec}
    \begin{aligned}
      \E_{X}^\psi[W] &= \bbetas{\psi}{Y}(W)
      + (\bbetas{\psi}{Y}\otimes \E^\psi_{X}) [\rdelta_{X}(W)]\\
      &= \bbetas{\psi}{Y}(W) + 
      (\E^\psi_{X}\otimes\bbetas{\psi}{Y})[\ldelta_{X}(W)].
    \end{aligned}
  \end{equation}
\end{theorem}

\begin{remark}
  If one is interested in stating the multivariate version of the previous theorem, the corresponding partial divided power derivations  to be considered are given by 
  \begin{align}
    \rdelta_{\alg{X}}(P) &= \sum (1\otimes X_i)\partial_{X_i}(P),
    \\
    \ldelta_{\alg{X}}(P) &= \sum (X_i\otimes 1)\partial_{X_i}(P),
  \end{align}
  for any $P\in \mathbb{C}\langle \alg{S}\rangle$, where the alphabet is
  $\alg{S}=\alg{X}\cup\alg{Y}$
  and the sums run over the letters  $X_i\in\alg{X}$.
\end{remark}

We now turn to the extension of the preceding results to the $c$-free
setting.
Since $X$ and $Y$ are $c$-free, they are free with respect to
$\psi$. In addition to the conditional expectation
$E^\psi_X$ from Proposition~\ref{prop:freecondexpX}
we will now construct an analogous map $\rE_X^\varphi$
with respect to $\varphi$, which,
as we have seen in Section~\ref{ssec:failurepositive},
is not necessarily positive and only satisfies 
the right module property.
\begin{definition}\label{def:rightcondexp}
  Let $\varphi:\C\langle X,Y\rangle \to \C$ be a noncommutative
  distribution. A \emph{right quasi-conditional expectation} for $\varphi$ is a
  linear map $E:\C\langle X,Y\rangle \to \C\langle X\rangle$ with the following
  properties: 
  \begin{enumerate}[label=(\roman*)]
   \item
    \label{def:it:invariance}
    \emph{Invariance.}
    For any polynomial $P(X,Y)\in \C\langle X,Y\rangle$ we have
    \[\varphi(\E[P(X,Y) ])=\varphi(P(X,Y)).\] 
   \item
    \label{def:it:rightmodular}
    \emph{Right modular property.}
    For any polynomial $P(X,Y)\in \C\langle X,Y\rangle$ and polynomial $Q(X)$ we have 
    \[
      \E[P(X,Y) Q(X)]=\E[P(X,Y)]Q(X)
      .\]
  \end{enumerate}
\end{definition}
At the time of this writing we are not able to provide a full formula
for $\rE_X^\varphi$ like \eqref{eq:Epsi:full} and rather have to rely on an
analogue of the recurrence \eqref{eq:condexprec} for its definition.
\begin{theorem} 
  \label{thm:rEphi:recurrence}
  Assume that  $X$ and $Y$ are $c$-free
  with respect to the pair of states $(\varphi,\psi)$
  on $\IC\langle X,Y\rangle$.
  Then the following recurrence defines a right quasi-conditional expectation
  for $\varphi$
  onto the subalgebra $\IC\langle X \rangle$:
  \begin{enumerate}[label=(\roman*)]
   \item
    $\rE_X^{\varphi}[P(X)]=P(X)$ for any polynomial $P(X)\in\IC\langle X\rangle$.
   \item 
    $\rE_X^{\varphi}[P(Y)]=\varphi(P(Y))$ for any polynomial $P(Y)\in\IC\langle Y\rangle$.
   \item
    If $W\in \C\langle X,Y\rangle$ is a monomial starting with $Y$,
    then
    \begin{equation}
      \label{eq:rEphi[YW]:rec}
      \rE_X^\varphi[W] = \bbetas{\varphi}{Y}(W) +
      \bigl(\bbetas{\varphi}{Y}\otimes\rE^\varphi_X\bigr)[\rdelta_XW].
    \end{equation}
   \item 
    If $W\in \C\langle X,Y\rangle$ is a monomial starting with $X$,  then
    \begin{equation}
      \label{eq:rEphiX[XW]}
      \rE_X^\varphi[W] =
      \IE_X^\psi[W] + \bigl(\bbetas{\varphi}{X}\otimes(\rE^\varphi_X-\IE_X^\psi)\bigr)[\rdelta_YW].
    \end{equation}
   \item In general, for any monomial $W\in \C\langle X,Y\rangle$, we have
    \begin{equation}
      \label{eq:rEphiX[W]:LX}
      \rE^\varphi_X[W]
      = \bbetas{\varphi}{Y}(W)
      +   (\bubbetas{\varphi}{Y}\otimes\rE^\varphi_X)[\rdelta_XW] 
      +   
      \bigl(
      \bubbetas{\varphi}{X}\otimes(\rE^\varphi_X-\IE^\psi_X)
      \bigr)[\rdelta_YW]
      + X\IE^\psi_X[L_X(W)],
    \end{equation}
    where $\bubbetas{\varphi}{}=\bbetas{\varphi}{}-\epsilon$ and $L_X$ is the left annihilation operator, i.e., 
    \[
      L_X(W) = \begin{cases}
        0 & \text{if $W=1$ or $W=YW'$ for some $W'\in \alg{M}$} 
        \\
        W' & \text{if $W=XW'$}
      \end{cases}
    \]
  \end{enumerate}
\end{theorem}

\begin{remark}
  \begin{enumerate}[label=(\roman*)]
   \item []
   \item Note that for possible lack of faithfulness this conditional
    expectation is not necessarily unique, however it is well-defined.
   \item 
    It is not difficult to see that writing out equations
    \eqref{eq:rEphi[YW]:rec} and \eqref{eq:rEphiX[XW]} results in
    \begin{equation}
      \label{eq:rEphi[Y0X1Y1]:rec}
      \rE_X^\varphi[Y_0X_1Y_1\dotsm X_nY_n] =
      \sum_{k=0}^n\beta^\varphi_{2k+1}(Y_0,X_1,Y_1,\dots,X_k,Y_k)\rE_X^\varphi[X_{k+1}Y_{k+1}\dotsm X_nY_n]
    \end{equation}
    and
    \begin{multline}
      \rE^\varphi_X[X_0Y_1X_1Y_2\dotsm Y_n]
      \\
      =         E_X^\psi[X_0Y_1X_1\dotsm Y_n]
      + \sum_{k=0}^{n-1}
      \beta^{\varphi}_{2k+1}(X_0,Y_1,X_1,Y_2,\dots, X_k)
      \rE_X^\varphi [Y_{k+1}X_{k+1}\dotsm Y_n]
      \\
      - \sum_{k=0}^{n-1} \beta^\varphi_{2k+1}(X_0,Y_1,X_1,\ldots, X_k)\,
      \E_X^\psi[Y_{k+1}X_{k+1}\dotsm Y_n]
    \end{multline}
    respectively, where $X_0,\ldots,X_n$ are non-constant monomials in $X$ and
    $Y_0,\ldots,Y_n$ are non-constant monomials in $Y$.  In particular,
    equation
    \eqref{eq:rEphi[Y0X1Y1]:rec} is identical to the recurrence in the free
    case \eqref{eq:Epsi:recurrence}.
\end{enumerate}
\end{remark}

\begin{remark}
  Rearranging identity   \eqref{eq:rEphiX[XW]}, we observe that for any monomial $W$ starting with  $X$, the element $\widetilde{W}= W -
  (\bbetas{\varphi}{X}\otimes \id)[\rdelta_YW]$ satisfies
  $$
  \rE^\varphi_X[\widetilde{W}]=  \IE^\psi_X[\widetilde{W}]
  .
  $$
\end{remark}

The following lemma connects the $c$-free quasi-conditional expectation
to the free conditional expectation and is the key to the proof of the preceding theorem.
\begin{lemma}
  If $W\in \C\langle X,Y\rangle$ is a monomial starting with $X$,  then
  \begin{equation}
    \varphi(\E^\psi_X[W]) = \bbetas{\varphi}{}(W) + \bigl(\bbetas{\varphi}{}\otimes(\varphi\circ \E^\psi_X)\bigr)[\rdelta_YW];
  \end{equation}
  that is, if $W= X_0Y_1X_1\dotsm Y_nX_n$ where $X_0,\ldots,X_n$ are non-constant monomials in $X$ and $Y_0,\ldots,Y_n$ are non-constant monomials in $Y$, then
  \begin{equation}
    \varphi(\E_X^\psi[W])
    = \sum_{k=0}^{n} \beta^\varphi_{2k+1}(X_0,Y_1,X_1,\ldots, X_k)\,
    \varphi(\E_X^\psi[Y_{k+1}X_{k+1}\dotsm Y_nX_n]).
  \end{equation}
  If $X_n$ is constant, i.e., 
  if $W=X_0Y_1X_1\cdots X_{n-1}Y_n$, then the identity holds as well, but the
  term  $\bbetas{\varphi}{}(W)=0$ vanishes.
\end{lemma}
\begin{proof}
We use the full expansion   \eqref{eq:Epsi:full} of the conditional expectation and then apply recurrence \eqref{eq:recurrenceboolcum} for the Boolean cumulants as follows:
  \begin{multline*}
    \varphi(\IE_X^\psi[X_0Y_1X_1\dotsm Y_nX_n])\\
    \begin{aligned}
      &= \sum_{p=0}^{n}\sum_{0=i_0 < i_1<i_2<\dots< i_p=n}
      \varphi( X_{0}X_{i_1}\dotsm X_{i_{p-1}}X_n)\,
      \beta^\psi(W_1)\beta^\psi(W_2)\dotsm\beta^\psi(W_p)
      \\
      &= \sum_{p=0}^{n}    \sum_{0 < i_1<i_2<\dots< i_p= n}
      \sum_{k=0}^p \beta^\varphi_{k+1}(X_0,X_{i_1},\dots,X_{i_k})\,
      \varphi( 
      X_{i_{k+1}}\dotsm X_{i_{p-1}}X_n)\,
      \beta^\psi(W_1)\beta^\psi(W_2)\dotsm\beta^\psi(W_p)
      \\
      &=
      \begin{multlined}[t]
        \sum_{s=0}^{n}
        \sum_{\substack{0 < i_1<i_2<\dots< i_k=s\\
            s < i'_1<i'_2<\dots< i'_l<n}
        }
        \beta^\varphi_{k+1}(X_0,X_{i_1},\dots,X_{i_{k-1}}X_s)\,
        \beta^\psi(W_1)\beta^\psi(W_2)\dotsm    \beta^\psi(W_k)
        \\
        \times
        \varphi(X_{i'_1}\dotsm X_{i'_l}X_n)
        \beta^\psi(W'_0)\beta^\psi(W'_1)\dotsm    \beta^\psi(W'_l)
      \end{multlined}
      \\
      &= \sum_{s=0}^{n} \beta^\varphi(X_0Y_1X_1\dotsm X_s)\,
      \varphi(\E_X^\psi[Y_{s+1}X_{s+1}\dotsm Y_nX_n])
    \end{aligned}
  \end{multline*}
  where
  for $j\geq1$ the factors are
  $W_j=Y_{i_{j-1}+1}X_{i_{j-1}+1}\dotsm Y_{i_j}$,
  $W'_0=Y_{s+1}X_{s+1}\dotsm Y_{i'_1}$ and
  $W'_j=Y_{i'_{j}+1}X_{i'_{j}+1}\dotsm Y_{i'_{j+1}}$.
  At the end  we used
  VNRP \eqref{eq:VNRP:cfree} and 
  the full expansion   \eqref{eq:Epsi:full} of
  the conditional expectation $E^\psi$.
\end{proof}

\begin{proof}[Pictorial proof]
  \begin{align*}
    \varphi(E_X^\psi[X_0Y_1X_1\dotsm Y_nX_n])
    &=  \sum
\begin{tikzpicture}[anchor=base,baseline]
        \node at (-1,0.8) {$\varphi$};
        \draw (-1,-0.4)--(-1,0.75)--(7,0.75)--(7,-0.4);
        \draw (1,-0.4)--(1,0.75);
        \draw (3,-0.4)--(3,0.75);
        \draw (5,-0.4)--(5,0.75);
\node at (-1,-0.8){$X_0$};
        \node at (1,-0.8){$X_{i_1}$};
        \node at (3.1,-0.8){$X_{i_2}$};
        \node at (5.3,-0.8){$X_{i_{k-1}}$};
        \node at (7,-0.8){$X_{n}$};
\paczek{0.0,0.0}{\beta^\psi}
        \paczek{2,0.0}{\beta^\psi}
        \node at (4,0.0){$\cdots$};
        \paczek{6,0.0}{\beta^\psi}
      \end{tikzpicture}
    \\
    &=
      \begin{multlined}[t]
        \sum\sum_s
        \begin{tikzpicture}[anchor=base,baseline]
          \node at (-1,0.8) {$\beta^\varphi$};
          \draw (-1,-0.4)--(-1,0.75)--(7,0.75)--(7,-0.4);
          \draw (1,-0.4)--(1,0.75);
          \draw (3,-0.4)--(3,0.75);
          \draw (5,-0.4)--(5,0.75);
\node at (-1,-0.8){$X_0$};
          \node at (1,-0.8){$X_{i_1}$};
          \node at (3.1,-0.8){$X_{i_2}$};
          \node at (5.3,-0.8){$X_{i_{k-1}}$};
          \node at (7,-0.8){$X_{s}$};
\paczek{0.0,0.0}{\beta^\psi}
          \paczek{2,0.0}{\beta^\psi}
          \node at (4,0.0){$\cdots$};
          \paczek{6,0.0}{\beta^\psi}
        \end{tikzpicture}
        \\
        \times
        \begin{tikzpicture}[anchor=base,baseline]
          \paczek{0.0,0.0}{\beta^\psi}
        \end{tikzpicture}
        \begin{tikzpicture}[anchor=base,baseline]
          \node at (-1,0.8) {$\varphi$};
          \draw (-1,-0.4)--(-1,0.75)--(7,0.75)--(7,-0.4);
          \draw (1,-0.4)--(1,0.75);
          \draw (3,-0.4)--(3,0.75);
          \draw (5,-0.4)--(5,0.75);
\node at (-1,-0.8){$X_{i'_1}$};
          \node at (1,-0.8){$X_{i'_2}$};
          \node at (3.1,-0.8){$X_{i_2}$};
          \node at (5.3,-0.8){$X_{i_{k-1}}$};
          \node at (7,-0.8){$X_{n}$};
\paczek{0.0,0.0}{\beta^\psi}
          \paczek{2,0.0}{\beta^\psi}
          \node at (4,0.0){$\cdots$};
          \paczek{6,0.0}{\beta^\psi}
        \end{tikzpicture}
      \end{multlined}
\end{align*}
  and we recognize both VNRP   \eqref{eq:VNRPphi:pic} and   the expansion of the
  conditional expectation  \eqref{eq:Epsi:full:pic} with $X_0=1$.
\end{proof}

\begin{proof}[Proof of Theorem~\ref{thm:rEphi:recurrence}]
  We proceed by induction on the block length of a word.
  For words of block length at most 1, both required properties follow directly from the definition.
  
  Let us prove invariance property~\ref{def:it:invariance}
  from Definition~\ref{def:rightcondexp}. For $W=X_1 Y_1 \cdots X_n Y_n X_n$ 
  Assume that both \eqref{eq:rEphi[YW]:rec} and \eqref{eq:rEphiX[XW]}
  hold for words of length smaller than the length of $W$.
  
  First we prove the invariance property~\ref{def:it:invariance}
  for a word which starts with $Y$.
  To this end we plug it into the  recurrence   \eqref{eq:rEphi[YW]:rec}
  and observe that the words appearing on the  right hand side of the latter
  have block length  strictly smaller than that of the original word $W$ and thus
  the induction hypothesis applies:
  \begin{align*}
    \varphi(\rE_X^\varphi[Y_1X_1Y_2\dotsm Y_nX_n])
    &=
      \varphi\bigl(\sum_{k=1}^n
      \beta^{\varphi}_{2k-1}(Y_1,X_1,Y_2,\dots, Y_k)
      \rE_X^\varphi [X_{k}Y_{k+1}\dotsm Y_nX_n]\bigr)\\
    &=\sum_{k=1}^n
      \beta^{\varphi}_{2k+1}(Y_1,X_1,Y_2,\dots, Y_k)\,
      \varphi(X_{k}Y_{k+1}\dotsm Y_nX_n) \\
    &=\varphi(Y_1X_1Y_2\dotsm Y_nX_n),  	
  \end{align*}
  where in the final line we used the fact that
  according to property (CAC) from Remark and Definition~\ref{remdef:fCAC}
  every other cumulant vanishes 
  in the  recurrence~\eqref{eq:recurrenceboolcum}.
For a word starting with $X$ the analogous method leaves us with
  \begin{multline*}
    \varphi
    (\rE_{X}[X_0Y_1X_1Y_2\dotsm Y_nX_n])\\
    =
\sum_{k=0}^{n-1}
        \beta^{\varphi}_{2k+1}(X_0,Y_1,X_1,Y_2,\dots, X_k)\,
        \varphi(\rE_X^\varphi [Y_{k+1}X_{k+1}\dotsm Y_n]X_n)
+ \varphi(\IE_X^\psi[X_0Y_1X_1\dotsm Y_n]X_n)
         \\
        - \sum_{k=0}^{n-1} \beta^\varphi_{2k+1}(X_0,Y_1,X_1,\ldots,X_k)\,
        \varphi(\E_X^\psi[Y_{k+1}X_{k+1}\dotsm Y_n]X_n).
\end{multline*}
From the preceding lemma we infer that the last two terms combine to
    the cumulant $\beta^{\varphi}_{2n+1}(X_0,Y_1,X_1,Y_2,\dots, X_n)$
and applying the induction hypothesis to the first term we arrive
    at the Boolean recurrence   \eqref{eq:recurrenceboolcum} and conclude
\begin{align*}
      \varphi
    (\rE^\varphi_{X}[X_0Y_1X_1Y_2\dotsm Y_nX_n])
    &=
    \begin{multlined}[t]
      \sum_{k=0}^{n-1}
      \beta^{\varphi}_{2k+1}(X_0,Y_1,X_1,Y_2,\dots, X_k)
      \,
      \varphi( Y_{k+1}X_{k+1}\dotsm Y_nX_n)
      \\
      + \beta^{\varphi}_{2n+1}(X_0,Y_1,X_1,Y_2,\dots, Y_n,X_n)
    \end{multlined}
    \\ 
    &=\varphi(X_0Y_1X_1Y_2\dotsm Y_nX_n). 
  \end{align*}

In order to prove \eqref{eq:rEphiX[W]:LX}, observe that if $W\in \alg{M}$
  is a monomial starting with $X$, then by definition of $\beta_Y^{b,\varphi}$
  every term  of  \eqref{eq:rEphi[YW]:rec} except the last one is annihilated
  and we have
  \[\rE_{X}^\varphi[W] = \bigl(\beta_Y^{b,\varphi}\otimes \rE_X^{\varphi} \bigr)[1\otimes W].
  \]
  In particular, if we write $\bubbetas{\varphi}{Y} = \bbetas{\varphi}{Y} - \epsilon$, we have that
  \begin{equation}
    \label{eq:monomialX}
    0 = \bubbetas{\varphi}{Y}(W) + \bigl( \bubbetas{\varphi}{Y} \otimes \rE_X^\varphi \bigr)[\rdelta_XW].
  \end{equation}
  On the other hand, from \eqref{eq:rEphiX[XW]} we observe that for any $W\in \alg{M}$:
  \begin{align}
    \IE_X^\psi[W] +
    \bigl(\bbetas{\varphi}{X}\otimes(\rE^\varphi_X-\IE_X^\psi)\bigr)[\rdelta_YW]
    \nonumber
    &= 
      \bigl(\bubbetas{\varphi}{X}\otimes(\rE^\varphi_X-\IE_X^\psi)\bigr)[\rdelta_YW] + (\epsilon\otimes \rE_X^\varphi)[\rdelta_YW] \\ &\phantom{xxx}+\IE_X^\psi[W]- (\epsilon\otimes \IE_X^\psi)[\rdelta_YW].\label{eq:monomialY}
  \end{align}
  Now, from the definition of $\rdelta_YW$ it is easy to see that
  \[
    (\epsilon\otimes \rE^\varphi_X)[\rdelta_YW]
    =
    \begin{cases}
      0 & \text{if $W=XW'$ or $W=1$} 
      \\
      \rE_X^\varphi[W] & \text{if $W=YW'$}
    \end{cases}
  \]
  and
  \begin{align*}
    \IE^\psi_X[W] 
    - (\epsilon\otimes \IE^\psi_X)[\rdelta_YW]
    &=
      \begin{cases}
        \IE^\psi_X[W] & \text{if $W=XW'$ or $W=1$}
        \\
        0 & \text{if $W=YW'$}
      \end{cases}
    \\
    &= \epsilon(W) + X\E^\psi_X[L_XW].
  \end{align*}
  Since $ \bigl(\bubbetas{\varphi}{X}\otimes(\rE^\varphi_X-\IE_X^\psi)\bigr)[\rdelta_YW] =0$ when $W$ is a monomial starting with $X$, we get that \eqref{eq:rEphiX[XW]} also holds for any $W\in \alg{M}$. Finally, we can combine the above equations so that, for any $W\in \alg{M}$ we conclude that
  \begin{align*}
    \rE^\varphi_X[W]
&= \bubbetas{\varphi}{Y}(W)
        +   (\bubbetas{\varphi}{Y}\otimes\rE^\varphi_X)[\rdelta_XW] 
        +   
        \bigl(
        \bubbetas{\varphi}{X}\otimes(\rE^\varphi_X-\IE^\psi_X)
        \bigr)[\rdelta_YW]
        + \IE^\psi_X[W] 
        - \epsilon\otimes \IE^\psi_X[\rdelta_YW]
    \\
      &= \bbetas{\varphi}{Y}(W)+   (\bubbetas{\varphi}{Y}\otimes\rE^\varphi_X)[\rdelta_XW] 
        +   
        \bigl(
        \bubbetas{\varphi}{X}\otimes(\rE^\varphi_X-\IE^\psi_X)
        \bigr)[\rdelta_YW]
        + X\IE^\psi_X[L_XW].
  \end{align*}
\end{proof}

The following result from \cite{LehnerSzpojankowski:2023:freeint1} explains how to compute the conditional expectation of the resolvent for free random variables $X$ and $Y$.

\begin{proposition}
  \label{prop:6.13}
  Assume that the random variables $X$ and $Y$ are free in the
  noncommutative probability  space 
  $(\mathbb{C}\langle X,Y\rangle,\psi)$.
  Then with the notations from Proposition~\ref{lem:SeveralEquations} the
  conditional expectation of the resolvent
  $\bPsi = (I - z(AX+BY))^{-1}$ 
  is given by 
  \begin{align}
    \label{eq:EpsiX[Psi]=bbetapsi1}
    \IE^\psi_X[\bPsi]
    &= \bigl(
      I-z\bbetas{\psi}{Y}(\bPsi)AX
      \bigr)^{-1}
      \bbetas{\psi}{Y}(\bPsi)
    \\
    \label{eq:EpsiX[Psi]=bbetapsi2}
    &= \bbetas{\psi}{Y}(\bPsi)
      \bigl(
      I-zAX\bbetas{\psi}{Y}(\bPsi)
      \bigr)^{-1}\\
    &=(I-zAX-zBF_Y)^{-1}.
      \label{eq:EpsiX[Psi]=(I-zAXpzBFY)}
  \end{align}
\end{proposition}

\begin{proof}
  Using \eqref{eq:condexprec}, we have
  \begin{align*}
    E_X^\psi[\bPsi] &= \beta_Y^{b,\psi}(\bPsi) + z(\beta_Y^{b,\psi}\otimes E_X^\psi)[\bPsi\odot AX\bPsi]
    \\ &= \beta_Y^{b,\psi}(\bPsi)\bigl(I + zAXE_X^\psi[\bPsi]\bigr).
  \end{align*}
  Hence
  \[
    E_X^\psi[\bPsi] = \bigl(I- z \beta_Y^{b,\psi}(\bPsi)AX\bigr)^{-1}\beta_Y^{b,\psi}(\bPsi).
  \]
\end{proof}

\begin{remark}
  In the framework of the previous proposition, we have that the respective
  \emph{subordination functions}   \cite{Biane98}   are given by
  \begin{align*}
    \omega_X(1/z) &= (z\bbetas{\psi}{Y}(\bPsi))^{-1},\\
    \omega_Y(1/z) &= (z\bbetas{\psi}{X}(\bPsi))^{-1}.
  \end{align*}
\end{remark}

The last theorem of this section generalizes Proposition~\ref{prop:6.13} to the c‑free setting, i.e., it shows how to compute the right quasi‑conditional expectation for c‑free elements $X$ and $Y$.

\begin{theorem}
Let $(\mathbb{C}\langle X,Y\rangle,\varphi,\psi)$ be a two-state
  noncommutative probability space such that $X$ and $Y$ are $c$-free.  Assume
  that the resolvent of a given polynomial $P \in \mathbb{C}\langle X,Y\rangle$
  of degree $m$ has a linearization given by $\Psi = (1-z^mP)^{-1} = u^t\bPsi
  v$, where $$\bPsi = (I-z(AX+BY))^{-1},$$ with $A,B\in
  M_N(\mathbb{C})$
  and $u,v\in \mathbb{C}^N$.
  Then
  \begin{align}
    \rE_X^\varphi[\bPsi]
    &=
      \bigl(
      I-zAF_X^\varphi-zBF_Y^\varphi
      \bigr)^{-1}
      \bigl(
      I  -
      zAF_X^\varphi-zBF_Y
      \bigr)
      (I-zAX-zBF_Y)^{-1}.
      \label{eq:AX+BY:rEphiX[Psi]=AFX} 
  \end{align}
\end{theorem}

\begin{proof}
  We start by proving an auxiliary formula:
  \begin{align}
    \rE_X^\varphi[\bPsi]
    &=
      \bbetas{\varphi}{Y}(\bPsi)
      \bigl(
      I - \bubbetas{\varphi}{X}(\bPsi)
      \bubbetas{\varphi}{Y}(\bPsi)
      \bigr)^{-1}
      \bigl(
      I
      -
      \bubbetas{\varphi}{X}(\bPsi)
      \bubbetas{\psi}{Y}(\bPsi)
      \bigr)
      \bigl(
      I-zAX\bbetas{\psi}{Y}(\bPsi)
      \bigr)^{-1}.
      \label{eq:AX+BY:rEphiX[Psi]=bbeta*res}
  \end{align}
  First it follows from 
  \eqref{eq:DA-1}
  that
  $$
  \rdelta_X\bPsi = z\bPsi A\odot X\bPsi,
  \qquad
  \rdelta_Y\bPsi = z\bPsi B\odot Y\bPsi.
  $$
  Secondly we apply  $\rE^\varphi_X$ to the resolvent identity
  \eqref{eq:resolventIdentity} and obtain
  \begin{equation}
    \label{eq:AX+BY:resid}
    \rE^\varphi_X[\bPsi] = I + zA\rE^\varphi_X[X\bPsi] + zB\rE^\varphi_X[Y\bPsi]
    .
  \end{equation}
  Next we apply the recursions from Theorem~\ref{thm:rEphi:recurrence} to each
  term of    \eqref{eq:AX+BY:resid}.
  More precisely, the recurrence  in Equation \eqref{eq:rEphiX[W]:LX} yields
  \begin{equation}
    \label{eq:AX+BY:rEphiX[Psi]:rec}
    \rE^\varphi_X[\bPsi] = \bbetas{\varphi}{Y}(\bPsi)
    + z  \bubbetas{\varphi}{Y}(\bPsi)A\rE^\varphi_X[X\bPsi]
    + z \bubbetas{\varphi}{X}(\bPsi)B
    \bigl(
    \rE^\varphi_X[Y\bPsi] - E^\psi_X[Y\bPsi]
    \bigr) + zAX E^\psi_X[\bPsi]
  \end{equation}
  Now apply recurrence \eqref{eq:rEphi[YW]:rec} to obtain
  \begin{align*}
    \rE^\varphi_X[Y\bPsi]
    &= \bbetas{\varphi}{Y}(Y\bPsi)
      + z\bbetas{\varphi}{Y}(Y\bPsi)A  \rE^\varphi_X[X\bPsi]\\
    &= \bbetas{\varphi}{Y}(Y\bPsi)
      \bigl(
      I + zA\rE^\varphi_X[X\bPsi]
      \bigr)\\
    &= \bbetas{\varphi}{Y}(Y\bPsi)
      \bigl(
      \rE^\varphi_X[\bPsi] - zB\rE^\varphi_X[Y\bPsi]
      \bigr),
  \end{align*}
  where we have used \eqref{eq:AX+BY:resid} in the last equality. Thus, it follows that
  \begin{equation*}
    \bigl(I+z\bbetas{\varphi}{Y}(Y\bPsi)B\bigr)  \rE^\varphi_X[Y\bPsi] = \bbetas{\varphi}{Y}(Y\bPsi)    \rE^\varphi_X[\bPsi]
  \end{equation*}
  and therefore
  \begin{equation}
    \label{eq:AX+BY:rEphiX[YPsi]}
    \rE^\varphi_X[Y\bPsi] =  
    \bigl(
    I+z\bbetas{\varphi}{Y}(Y\bPsi)B
    \bigr)^{-1}
    \bbetas{\varphi}{Y}(Y\bPsi)    \rE^\varphi_X[\bPsi].
  \end{equation}
  Furthermore, the previous equation implies that
  \begin{align*}
    zB  \rE^\varphi_X[Y\bPsi]
    &=    \bigl(
      I+zB\bbetas{\varphi}{Y}(Y\bPsi)
      \bigr)^{-1}
      zB\bbetas{\varphi}{Y}(Y\bPsi)    \rE^\varphi_X[\bPsi]
    \\
    &=    \bigl(
      I-(I+zB\bbetas{\varphi}{Y}(Y\bPsi))^{-1}
      \bigr)
      \rE^\varphi_X[\bPsi]
    \\
    &= \bigl(I-\bbetas{\varphi}{Y}(\bPsi)^{-1} \bigr)
      \rE^\varphi_X[\bPsi]
  \end{align*}
  where the last equality follows by applying $\bbetas{\varphi}{Y}$ to the
  resolvent identity \eqref{eq:resolventIdentity}. In particular, we get
  \[
    \rE_{X}^{\varphi}[\bPsi] - zB\rE_X^\varphi[Y\bPsi] = \beta_Y^{b,\varphi}(\bPsi)^{-1}\rE_X^{\varphi}[\bPsi].
  \]
  On the other hand, from the above equation and \eqref{eq:AX+BY:resid} we obtain
  \begin{align*}
    zA\rE^\varphi_X[X\bPsi]
    &= \rE^\varphi_X[\bPsi] - zB\rE^\varphi_X[Y\bPsi]  - I \\
    &=    \bbetas{\varphi}{Y}(\bPsi)^{-1}    \rE^\varphi_X[\bPsi] - I
      .
  \end{align*}
  Plugging everything into
  \eqref{eq:AX+BY:rEphiX[Psi]:rec}
  we obtain
  \begin{multline}
    \label{eq:rEvarphi_XPlugging}
    \rE^\varphi_X[\bPsi]
    = \bbetas{\varphi}{Y}(\bPsi)
    +   \bubbetas{\varphi}{Y}(\bPsi)
    \bigl(
    \bbetas{\varphi}{Y}(\bPsi)^{-1}    \rE^\varphi_X[\bPsi] - I
    \bigr)
    \\
    +  \bubbetas{\varphi}{X}(\bPsi)
    \bigl(
    \bigl(
    I-\bbetas{\varphi}{Y}(\bPsi)^{-1} 
    \bigr)
    \rE^\varphi_X[\bPsi]   
    -
    zBE^\psi_X[Y\bPsi]
    \bigr)
    + zAX E^\psi_X[\bPsi]    
  \end{multline}
  Now we focus on $zBE_X^\psi[Y\bPsi]$.  Applying $\E_X^\psi$ to the resolvent identity~\eqref{eq:resolventIdentity}, we obtain
  $$
  zBE^\psi_X[Y\bPsi] = E^\psi_X[\bPsi] - zAXE^\psi_X[\bPsi] - I,
  $$
  so that we can substitute it into \eqref{eq:rEvarphi_XPlugging} as
  \begin{multline*}
    \rE^\varphi_X[\bPsi]
    = \bbetas{\varphi}{Y}(\bPsi)
    +   \bubbetas{\varphi}{Y}(\bPsi)
    \bigl(
    \bbetas{\varphi}{Y}(\bPsi)^{-1}    \rE^\varphi_X[\bPsi] - I
    \bigr)
    +  \bubbetas{\varphi}{X}(\bPsi)
    \bigl(
    I-\bbetas{\varphi}{Y}(\bPsi)^{-1} 
    \bigr)
    \rE^\varphi_X[\bPsi]
    \\
    -
    \bubbetas{\varphi}{X}(\bPsi)
    \bigl(
    (I-zAX)E^\psi_X[\bPsi]-I
    \bigr)
    + zAX E^\psi_X[\bPsi] .
  \end{multline*}
  Next, we can rearrange the previous equation as follows:
  \begin{multline}
    \label{eq:rEphiX[Psi]:equation}
    \bigl(
    I-
    \bubbetas{\varphi}{Y}(\bPsi)
    \bbetas{\varphi}{Y}(\bPsi)^{-1}
    -
    \bubbetas{\varphi}{X}(\bPsi)
    \bigl(
    I-\bbetas{\varphi}{Y}(\bPsi)^{-1} 
    \bigr)
    \bigr)
    \rE^\varphi_X[\bPsi]
    \\
    \begin{aligned}
&= \bbetas{\varphi}{Y}(\bPsi)
      -   \bubbetas{\varphi}{Y}(\bPsi)
      -
      \bubbetas{\varphi}{X}(\bPsi)
      \bigl(
      (I-zAX)E^\psi_X[\bPsi]
      \bigr)
      + \bubbetas{\varphi}{X}(\bPsi)
      + zAX E^\psi_X[\bPsi]    
      \\
&= I 
      + \bubbetas{\varphi}{X}(\bPsi)
      +
      \bigl(
      \bbetas{\varphi}{X}(\bPsi)zAX
      -
      \bubbetas{\varphi}{X}(\bPsi)
      \bigr)
      E^\psi_X[\bPsi] 
      \\
&= \bbetas{\varphi}{X}(\bPsi)
      +
      \bigl(
      \bbetas{\varphi}{X}(\bPsi)zAX
      -
      \bubbetas{\varphi}{X}(\bPsi)
      \bigr)
      \bbetas{\psi}{Y}(\bPsi)
      \bigl(
      I-zAX\bbetas{\psi}{Y}(\bPsi)
      \bigr)^{-1} \\
      &=
      \begin{multlined}[t]
        \bbetas{\varphi}{X}(\bPsi)
        +
        \bbetas{\varphi}{X}(\bPsi)
        \bigl(
        \bigl(
        I-zAX\bbetas{\psi}{Y}(\bPsi)
        \bigr)^{-1}
        -I
        \bigr)\\
        -
        \bubbetas{\varphi}{X}(\bPsi)
        \bbetas{\psi}{Y}(\bPsi)
        \bigl(
        I-zAX\bbetas{\psi}{Y}(\bPsi)
        \bigr)^{-1}
      \end{multlined}
      \\
      &=
      \bigl(
      \bbetas{\varphi}{X}(\bPsi)
      -
      \bubbetas{\varphi}{X}(\bPsi)
      \bbetas{\psi}{Y}(\bPsi)
      \bigr)
      \bigl(
      I-zAX\bbetas{\psi}{Y}(\bPsi)
      \bigr)^{-1}
      \\
      &=
      \bigl(
      I
      -
      \bubbetas{\varphi}{X}(\bPsi)
      \bubbetas{\psi}{Y}(\bPsi)
      \bigr)
      \bigl(
      I-zAX\bbetas{\psi}{Y}(\bPsi)
      \bigr)^{-1}
    \end{aligned}
  \end{multline}
  where in the third equality, we used the expression   \eqref{eq:EpsiX[Psi]=bbetapsi2} for $E^\psi_X[\bPsi]$.
  Finally, the prefactor on the left-hand side of \eqref{eq:rEphiX[Psi]:equation} can be written as
  \begin{multline*}
    I-
    \bubbetas{\varphi}{Y}(\bPsi)
    \bbetas{\varphi}{Y}(\bPsi)^{-1}
    -
    \bubbetas{\varphi}{X}(\bPsi)
    \bigl(
    I-\bbetas{\varphi}{Y}(\bPsi)^{-1} 
    \bigr)
    \\
    \begin{aligned}
      &=
      \bigl(
      \bbetas{\varphi}{Y}(\bPsi)
      - \bubbetas{\varphi}{Y}(\bPsi)
      - \bubbetas{\varphi}{X}(\bPsi)
      \bbetas{\varphi}{Y}(\bPsi)
      + \bubbetas{\varphi}{X}(\bPsi)
      \bigr)
      \bbetas{\varphi}{Y}(\bPsi)^{-1}
      \\
      &=
      \bigl(
      I - \bubbetas{\varphi}{X}(\bPsi)
      \bubbetas{\varphi}{Y}(\bPsi)
      \bigr)
      \bbetas{\varphi}{Y}(\bPsi)^{-1}
    \end{aligned}
  \end{multline*}
  and we conclude 
  \eqref{eq:AX+BY:rEphiX[Psi]=bbeta*res}.
  
  Now we proceed with the proof of 	\eqref{eq:AX+BY:rEphiX[Psi]=AFX}.
  First, observe that 
  \begin{align*}
    \bubbetas{\varphi}{X}(\bPsi)
    &= \bbetas{\varphi}{X}(\bPsi) - I\\
    &= (I-zAF_X^\varphi)^{-1} - I\\
    &= (I-zAF_X^\varphi)^{-1}zAF_X^\varphi\\
    &= \bbetas{\varphi}{X}(\bPsi)zAF_X^\varphi.
  \end{align*}
On the other hand, using that $$\beta_Y^{b,\varphi}(\bPsi)^{-1} = I - zBF_Y^\varphi,$$the coefficient of $\rE_X^\varphi[\bPsi]$ in \eqref{eq:rEphiX[Psi]:equation}  can be written as
  \begin{align*}
    \MoveEqLeft[4] I - \bubbetas{\varphi}{Y}(\bPsi) \bbetas{\varphi}{Y}(\bPsi)^{-1}
    - \bubbetas{\varphi}{X}(\bPsi) \bigl( I - \bbetas{\varphi}{Y}(\bPsi)^{-1} \bigr) \\
    &= I - zBF_Y^\varphi - zAF_X^\varphi(I-zAF_X^\varphi)^{-1}zBF_Y^\varphi \\
    &= I - (I-zAF_X^\varphi)^{-1}zBF_Y^\varphi \\
    &= (I-zAF_X^\varphi)^{-1} \bigl( I - zAF_X^\varphi - zBF_Y^\varphi \bigr) \\
    &= \bbetas{\varphi}{X}(\bPsi) M^\varphi_{AX+BY}(z)^{-1},
  \end{align*}
  where we have applied Equations \eqref{eq:HXphi=1-zBFXphi} and \eqref{eq:AX+BY:phi(Psi)} in the last step. Finally, the right-hand side in the third equality in \eqref{eq:rEphiX[Psi]:equation} is
  \begin{multline*}
    \bbetas{\varphi}{X}(\bPsi)
    +
    \bigl(
    \bbetas{\varphi}{X}(\bPsi)zAX
    -
    \bubbetas{\varphi}{X}(\bPsi)
    \bigr)
    \bbetas{\psi}{Y}(\bPsi)
    \bigl(
    I-zAX\bbetas{\psi}{Y}(\bPsi)
    \bigr)^{-1}
    \\
    \begin{aligned}[t]
      &=   \bbetas{\varphi}{X}(\bPsi)
      +
      \bigl(
      \bbetas{\varphi}{X}(\bPsi)zAX
      -
      \bbetas{\varphi}{X}(\bPsi)zAF_X^\varphi
      \bigr)
      \bbetas{\psi}{Y}(\bPsi)
      \bigl(
      I-zAX\bbetas{\psi}{Y}(\bPsi)
      \bigr)^{-1}
      \\
      &=   \bbetas{\varphi}{X}(\bPsi)
      \bigl(
      I-zAX\bbetas{\psi}{Y}(\bPsi)
      + zAX  \bbetas{\psi}{Y}(\bPsi)
      -
      zAF_X^\varphi  \bbetas{\psi}{Y}(\bPsi)
      \bigr)
      \bigl(
      I-zAX\bbetas{\psi}{Y}(\bPsi)
      \bigr)^{-1}
      \\
      &=   \bbetas{\varphi}{X}(\bPsi)
      \bigl(
      I  -
      zAF_X^\varphi(I-zBF_Y)^{-1}
      \bigr)
      \bigl(
      I-zAX(I-zBF_Y)^{-1}
      \bigr)^{-1}
      \\
      &=
      (I-zAF_X^\varphi-zBF_Y^\varphi)^{-1}
      \bigl(
      I  -
      zAF_X^\varphi-zBF_Y
      \bigr)
      \bigl(
      I-zAX-zBF_Y
      \bigr)^{-1}.
      \\
    \end{aligned}
  \end{multline*}
  We conclude by using \eqref{eq:HXphi=1-zBFXphi}, \eqref{eq:AX+BY:phi(rEphiX[Psi])=MAFX} and \eqref{eq:EpsiX[Psi]=(I-zAXpzBFY)}.
\end{proof}

\begin{remark}
  By using \eqref{eq:EpsiX[Psi]=bbetapsi2}, we can write \eqref{eq:AX+BY:rEphiX[Psi]=bbeta*res}
  as
  \begin{align}
    \rE_X^\varphi[\bPsi]
    &= \bbetas{\varphi}{Y}(\bPsi)
      \bigl(
      I - \bubbetas{\varphi}{X}(\bPsi)
      \bubbetas{\varphi}{Y}(\bPsi)
      \bigr)^{-1}
      \bigl(
      I
      -
      \bubbetas{\varphi}{X}(\bPsi)
      \bubbetas{\psi}{Y}(\bPsi)
      \bigr)
      \bbetas{\psi}{Y}(\bPsi)^{-1}
      E^\psi_X[\bPsi]
      \label{eq:AX+BY:rEphX[Psi]=prefactorEpsiX[Psi]}.    
  \end{align}
\end{remark}

\begin{proposition}
  With the notation and hypotheses of Theorem~\ref{thm:cfreeConvolution}, we have
  \begin{equation}
    \varphi(E^\psi_X[\bPsi])
    = (I-zAF^\varphi_X-zBF_Y)^{-1}
    \label{eq:AX+BY:phi(rEphiX[Psi])=MAFX}.
  \end{equation}
  As a consequence we obtain
  \begin{align}
    \rE_X^\varphi[\bPsi]
    &= M^\varphi_{AX+BY}(z)\varphi(E^\psi_X[\bPsi])^{-1}E^\psi_X[\bPsi].
      \label{eq:AX+BY:rEphiX[Psi]=MAX+BY}       
  \end{align}
  
\end{proposition}

\begin{proof}
  We use the identities in Proposition~\ref{lem:SeveralEquations} and Proposition~\ref{prop:6.13}. In particular, using \eqref{eq:EpsiX[Psi]=(I-zAXpzBFY)} and \eqref{eq:HYAdditive}, we have
  \begin{align*}
    \varphi(E^\psi_X[\bPsi])
    &= \varphi\bigl((I-zAX-zBF_Y)^{-1}\bigr)\\
    &= (I-zBF_Y)^{-1}\varphi\bigl((I-zAX(I-zBF_Y)^{-1})^{-1}\bigr)   \\
    &= H_Y\varphi((I-zAXH_Y)^{-1})   \\
    &= H_Y M^\varphi_X(zAH_Y).
  \end{align*}
  Now, using that $M_X^\varphi(zAH_Y) = (I-\eta_X^\varphi(zAH_Y))^{-1}$ and that 
  \[
    \eta^\varphi_X(zAH_Y)=zA\tilde{\eta}^\varphi_X(zH_YA)H_Y=zAF^\varphi_XH_Y,
  \]we obtain
  \begin{align*}
    \varphi(E^\psi_X[\bPsi])
    &=H_Y
      \bigl(
      I-\eta^\varphi_X(zAH_Y)
      \bigr)^{-1}   \\
    &= H_Y(I-zAF^\varphi_XH_Y)^{-1}
    \\ &= \bigl((I-zAF^\varphi_XH_Y)H_Y^{-1}\bigr)^{-1}
    \\ &= \bigl( H_Y^{-1} - zAF^\varphi_X\bigr)^{-1}
    \\&= 
    \bigl(    I-zAF^\varphi_X-zBF_Y    \bigr)^{-1},
  \end{align*}
  where we used \eqref{eq:HYAdditive} in the last equality.
\end{proof}

\section{$c$-free multiplicative convolution}
\label{sec:sigmatransform}

Our goal in this section is to show, as an application of the developments presented in this paper, how to compute the distribution of the product of two $c$-free random variables. Moreover, we also reproduce the multiplicativity of the \emph{Popa-Wang $\Sigma$-transform} from \cite{PopaWang:2011:multiplicative}.

\begin{proposition}
  Let $(\mathbb{C}\langle X,Y\rangle,\varphi,\psi)$ be a two-state noncommutative probability space such that $X$ and $Y$ are $c$-free. Then the moment generating function of $XY$ with respect to $\varphi$ is given by
  \begin{equation}
    M_{XY}^\varphi(z) = \frac{1}{
      1- z\tilde{\eta}^\varphi_X(\omega_X(z))\tilde{\eta}^\varphi_Y(\omega_Y(z))
    },
    \label{eq:XY:MphiXY}
  \end{equation}
  where $\omega_X(z)$ and $\omega_X(Y)$ are the subordination functions for the free multiplicative convolution of $X$ and $Y$ with respect to $\psi$. Furthermore, if we set $\Sigma^\varphi = \tilde\eta^\varphi\circ(\eta^\psi)^{-1}$, then
  \begin{equation}
    \label{eq:PopaWang}
    \Sigma^\varphi_{XY} = \Sigma^\varphi_X\cdot \Sigma^\varphi_Y.
  \end{equation}
\end{proposition}
\begin{proof}
  In order compute the moment generating function $M_{XY}^\varphi(z)$, we consider 
  \begin{equation}
    \label{eq:ResolventMult}
    \bPsi := (1-zXY)^{-1}
    = 1+zXY\bPsi
  \end{equation}
  so that $M_{XY}^\varphi(z) =\varphi(\rE_X^\varphi[\bPsi])$. It is easy to compute the following block derivations:
  \begin{align*}
    \rdelta_X\bPsi &= z\bPsi\otimes XY\bPsi &    \rdelta_Y\bPsi &= z\bPsi X\otimes Y\bPsi \\
                   &= \bPsi\otimes(\bPsi-1)
  \end{align*}
Using resolvent identity in \eqref{eq:ResolventMult}, we can resort to the simpler recurrence
  \eqref{eq:rEphiX[XW]} in Theorem~\ref{thm:rEphi:recurrence} and obtain
  \begin{align*}
    \rE^\varphi_X[\bPsi]
    &= 1 +   z \rE^\varphi_X[XY\bPsi]\\
    &= 1+zXE^\psi_X[Y\bPsi] +
      z\bbetas{\varphi}{X}\otimes(\rE^\varphi_X-E^\psi_X)\rdelta_Y[XY\bPsi].
  \end{align*}
  A short calculation reveals 
  \[
    \rdelta_Y [XY\bPsi] = \bPsi
    X\otimes Y\bPsi
  \] 
  and thus
  \begin{equation}
    \label{eq:rEMult}
    \rE^\varphi_X[\bPsi]  = E^\psi_X[\bPsi]
    +
    z\bbetas{\varphi}{X}(\bPsi X)
    \bigl(
    \rE^\varphi_X[Y\bPsi] - E^\psi_X[Y\bPsi]
    \bigr)
  \end{equation}
The recurrence
  \eqref{eq:rEphi[YW]:rec}
  for $\rE^\varphi_X[Y\bPsi]$ in Theorem~\ref{thm:rEphi:recurrence} is identical to the one \eqref{eq:condexprec}  for  $E^\psi_X[Y\bPsi]$. This means:
  \begin{align}
    \rE^\varphi_X[Y\bPsi]\nonumber
    &= \bbetas{\varphi}{Y}(Y\bPsi) + z\bbetas{\varphi}{Y}(Y\bPsi)\rE^\varphi_X[XY\bPsi]\\
    &= \bbetas{\varphi}{Y}(Y\bPsi)\rE^\varphi_X[\bPsi]
      \label{eq:XY:EvarphiX[YPsi]}
  \end{align}
and
  \begin{equation}
    \label{eq:XY:EpsiX[YPsi]}
    E^\psi_X[Y\bPsi] = \bbetas{\psi}{Y}(Y\bPsi)E^\psi_X[\bPsi]
    .
  \end{equation}
  Therefore, substituting \eqref{eq:XY:EvarphiX[YPsi]} and \eqref{eq:XY:EpsiX[YPsi]} in \eqref{eq:rEMult}, we obtain
  \begin{equation*}
    \rE^\varphi_X[\bPsi]
    = E^\psi_X[\bPsi]
    +
    z\bbetas{\varphi}{X}(\bPsi X)
    \bigl(
    \bbetas{\varphi}{Y}(Y\bPsi)\rE^\varphi_X[\bPsi]
    -
    \bbetas{\psi}{Y}(Y\bPsi)E^\psi_X[\bPsi]
    \bigr),
  \end{equation*}
  which implies
  \begin{equation}
    \label{eq:XY:rEphiX[Psi]}
    \rE^\varphi_X[\bPsi]
    =
    \frac{
      1-z\bbetas{\varphi}{X}(\bPsi X)     \bbetas{\psi}{Y}(Y\bPsi)
    }{
      1-z\bbetas{\varphi}{X}(\bPsi X)     \bbetas{\varphi}{Y}(Y\bPsi)
    }
    E^\psi_X[\bPsi]
  \end{equation}
Now, we look at the factor $\E_X^\psi[\bPsi]$ in the right-hand side of the above equation. The free recurrence \eqref{eq:condexprec} together
  with   \eqref{eq:XY:EpsiX[YPsi]} yields
  \begin{align}
    E^\psi_X[\bPsi]
    &= 1 + zXE^\psi_X[Y\bPsi]\nonumber\\
    &= 1 + zX\bbetas{\psi}{Y}(Y\bPsi)E^\psi_X[\bPsi]\nonumber
      \intertext{and solving}
      E^\psi_X[\bPsi]    
    &=
      \bigl(
      1 - z\bbetas{\psi}{Y}(Y\bPsi)X
      \bigr)^{-1}.\label{eq:EXpsiPsiMult}
  \end{align}
It follows that
  \begin{equation}
    \label{eq:XY:omegaX}
    \omega_X(z) = z\bbetas{\psi}{Y}(Y\bPsi)
  \end{equation}
  is the \emph{subordination function}, i.e.,~$\omega_X(z)$ satisfies
  \[
    M_{XY}^\psi(z) = M_X^\psi(\omega_X(z)).
  \]
  Analogously, we have that
  \begin{equation}
    E^\psi_Y[\bPsi] =
    \bigl(
    1-z\bbetas{\psi}{X}(\bPsi X)Y
    \bigr)^{-1}
  \end{equation}
  and
  \begin{equation}
    \label{eq:XY:omegaY}
    \omega_Y(z) = z\bbetas{\psi}{X}(\bPsi X).
  \end{equation}
  
Our next step is to deduce a system of equations for the subordination functions $\omega_X(z)$ and $\omega_Y(z)$. First, observe that all terms in the expansion $\bPsi(z) = \sum_{n=0}^{\infty} (XY)^n z^n$ are alternating words, with each block of letters in the alternation having length one. In particular, this implies that $\beta_Y^{b,\psi}(Y\bPsi)=\beta_Y^{\delta,\psi}(Y\bPsi)$. Hence
  \[
    \omega_X(z) = z\beta_Y^{b,\psi}(Y\bPsi)= z\beta_Y^{\delta,\psi}(Y\bPsi).
  \]
  Now, since $\partial_Y(\bPsi) =z\bPsi X\otimes \bPsi$ and, we can show by induction that
  \[
    \partial_Y^n(Y\bPsi) = z^{n-1}1\otimes (\bPsi  X)^{\otimes (n-1)}\otimes \bPsi + z^nY\bPsi X \otimes (\bPsi X)^{\otimes (n-1)}\otimes \bPsi.
  \]
  Then, by \eqref{eq:IntroVNRP_for_betas:uni} we obtain
  \begin{align*}
    \omega_X(z)
    &= z\fbetas{\psi}{Y}(Y\bPsi)\\
    &= z\tilde{\eta}^\psi_Y(z\bbetas{\psi}{X}(\bPsi X))\\
    &= z\tilde{\eta}^\psi_Y(\omega_Y(z)).
      \intertext{Analogoulsy, we have}
      \omega_Y(z)
    &= z\tilde{\eta}^\psi_X(\omega_X(z)).
  \end{align*}
  A similar argument for $\bbetas{\varphi}{X}(\bPsi X)$ and $\bbetas{\varphi}{Y}(Y\bPsi)$ but now using \eqref{eq:ProductsasEntriesCFree} yields 
  \begin{align*}
    \bbetas{\varphi}{X}(\bPsi X)
    &=   \fbetas{\varphi}{X}(\bPsi X)\\
    &= \tilde{\eta}^\varphi_X(z\beta^ {b,\psi}_{Y}(Y\bPsi))\\
    &= \tilde{\eta}^\varphi_X(\omega_X(z)),
      \intertext{and}
      \bbetas{\varphi}{Y}(Y\bPsi)
    &=   \fbetas{\varphi}{Y}(Y\bPsi)\\
    &= \tilde{\eta}^\varphi_Y(z\beta^{b,\psi}_{X}(\bPsi X))\\
    &= \tilde{\eta}^\varphi_Y(\omega_Y(z)).
  \end{align*}
Finally, we can plug the above identities into   \eqref{eq:XY:rEphiX[Psi]} in order to obtain
  \begin{align*}
    M_{XY}^\varphi(z)
    &= \varphi(\rE^\varphi_X[\bPsi])\\
    &=   \frac{
      1-z\bbetas{\varphi}{X}(\bPsi X)     \bbetas{\psi}{Y}(Y\bPsi)
      }{
      1-z\bbetas{\varphi}{X}(\bPsi X)     \bbetas{\varphi}{Y}(Y\bPsi)
      }
      \varphi(E_X^\psi[\bPsi])\\
    &=   \frac{
      1-z\bbetas{\varphi}{X}(\bPsi X)     \bbetas{\psi}{Y}(Y\bPsi)
      }{
      1-z\bbetas{\varphi}{X}(\bPsi X)     \bbetas{\varphi}{Y}(Y\bPsi)
      }
      M_X^\varphi(z\bbetas{\psi}{Y}(Y\bPsi))\\
    &= \frac{1-z\tilde{\eta}^\varphi_X(\omega_X(z))\frac{\omega_X(z)}{z}
      }{
      1- z\tilde{\eta}^\varphi_X(\omega_X(z))\tilde{\eta}^\varphi_Y(\omega_Y(z))
      }
      M^\varphi_X(\omega_X(z)).
  \end{align*}
  Recalling that $M_X^\varphi(s) = (1 - \eta^\varphi_X(s))^{-1}$ and $\eta_X^\varphi(s) = s\tilde\eta_X^\varphi(s)$. By taking $s=\omega_X(z)$, we conclude \eqref{eq:XY:MphiXY}:
  \begin{equation*}
    M_{XY}^\varphi(z) = \frac{1}{
      1- z\tilde{\eta}^\varphi_X(\omega_X(z))\tilde{\eta}^\varphi_Y(\omega_Y(z))
    }.
  \end{equation*}
Now we prove \eqref{eq:PopaWang}. According to \cite[Theorem 3.2]{BelBerNewApproach}, it is known that
  \begin{align*}
    \eta^\psi_{XY}(z)
    &= \eta^\psi_{X}(\omega_X(z))\\
    &= \eta^\psi_{Y}(\omega_Y(z)).
  \end{align*}
  If we set $u=\eta^\psi_{XY}(z)$, then
  \begin{align*}
    \omega_X(z)
    &= (\eta^{\psi}_X)^{-1}(u)\\
    \omega_Y(z)
    &= (\eta^{\psi}_Y)^{-1}(u).
      \intertext{Also, we have}
      \eta^\psi_{XY}(z)
    &= \tilde{\eta}^\psi_X(\omega_X(z))\omega_X(z)\\
    &= \tilde{\eta}^\psi_X(\omega_X(z))z \tilde{\eta}^\psi_Y(\omega_Y(z)).
      \intertext{Hence}
      \tilde{\eta}^\psi_{XY}(z)
    &= \tilde{\eta}^\psi_X(\omega_X(z)) \tilde{\eta}^\psi_Y(\omega_Y(z))
      \intertext{so that}
      \tilde{\eta}^\psi_{XY}\circ(\eta^\psi_{XY})^{-1}(u)
    &= \tilde{\eta}^\psi_X\circ(\eta^\psi_{X})^{-1}(u)
      \cdot
      \tilde{\eta}^\psi_Y\circ(\eta^\psi_{Y})^{-1}(u).
  \end{align*}
On the other hand, from \eqref{eq:XY:MphiXY} we infer that the $\tilde\eta$-transform of $XY$ with respect to $\varphi$ is given by
  \[
    \tilde{\eta}^\varphi_{XY}(z)
    = \tilde{\eta}^\varphi_X(\omega_X(z))\tilde{\eta}^\varphi_Y(\omega_Y(z)).
  \]
  By setting again $u = \eta_{XY}^\psi(z)$, we conclude that
  \begin{equation*}
    \tilde{\eta}^\varphi_{XY}\circ(\eta^\psi_{XY})^{-1}(u)
    =
    \tilde{\eta}^\varphi_{X}\circ(\eta^\psi_{X})^{-1}(u)
    \cdot
    \tilde{\eta}^\varphi_{Y}\circ(\eta^\psi_{Y})^{-1}(u)    
  \end{equation*}
  which reproduces \cite[Lemma~2.1]{PopaWang:2011:multiplicative} upon setting $\Sigma^\varphi = \tilde{\eta}^\varphi \circ (\eta^\psi)^{-1}$, and is precisely \eqref{eq:PopaWang}.
\end{proof}

\bibliographystyle{abbrv}

\end{document}